\documentclass[11pt]{article}
\input{macros}
\usepackage{mathrsfs}
\usepackage{todonotes}
\allowdisplaybreaks

\title{Free energy of Ising models under a
spectral condition}

\author{Andrea Montanari\thanks{Department of
Mathematics and Department of Statistics, Stanford University} \,\,
and \,\, Michael Ren\thanks{Department of Mathematics, Stanford
University}}
\date{\today}

\begin{document}

\maketitle

\begin{abstract}
A sequence of sparse weighted graphs $(G_N:N\ge 1)$ indexed by the
number of vertices $N$ is said to be left-convergent if all (suitably
weighted) subgraph counts converge to a limit as $N\to\infty$. This
notion generalizes in a natural way Benjamini-Schramm's definition of
local weak convergence. A broad research agenda aims at determining
which `global' graph properties are determined by left or local weak
convergence (in other words, which of these properties are in fact
local). We prove that, under a spectral condition on the weighted
adjacency matrix $\bA_N$ of graph $G_N$, the free energy density of
the Ising model on this weighted graph is continuous in the left
convergence topology (and hence is a local function).

Our proof uses the decomposition of the Ising measure as a log-concave
combination of product measures, and of the rapid mixing of Langevin dynamics for log-concave measures.

 As applications, we derive new limit theorems for the free energy
 density of spin glasses, antiferromagnets, and magnetization
 constrained ferromagnetic models, on locally tree-like graphs.
\end{abstract}

\tableofcontents

\section{Introduction}\label{sec:statement}

In this paper, we consider \emph{left-convergent} sequences of
weighted graphs. This notion of graph convergence was introduced in
\cite{BCLSV08} for sequences of dense graphs and further developed in
\cite{BCLSV12,BCKL13,BCG17,BCCZ19,BCCZ18}. In the case of sparse
graphs, left convergence is a generalization of local weak
convergence, as first introduced by Benjamini and Schramm \cite{BS01}
and subsequently studied by Aldous and co-authors \cite{AS04,AL07}.
(See also Remark \ref{rmk:ConvEquivalence} below.)

For our purposes, we will need a slight generalization of left
convergence that allows for both $\R$-valued weights on edges, to be
interpreted as coupling strengths, and $\R^k$-valued weights on
vertices, to be interpreted as external fields.

Consider a weighted graph $G$ on $N$ vertices $[N]=\{1,2,\ldots,N\}$
with edge weights $a_{ij}=a_{ji}\in\R$ for $i,j\in[N]$ and vertex
weights $(h_{1,i},\ldots,h_{k,i})\in\R^k$ for $i\in[N]$. We represent
this data by the tuple $(\AA,\hh_1,\ldots,\hh_k)$ for a symmetric
matrix $\AA=(a_{ij})_{i,j=1}^N\in\R^{N\times N}$ and vectors
$\hh_\ell=(h_{\ell,1},\ldots,h_{\ell,N})\in\R^N$.

Left-convergent sequences of graphs are defined by the convergence of
weighted homomorphism densities from all finite multigraphs, as will
be stated formally below. To accommodate our extension with external
fields, we define a \emph{test graph} $X$ as a multigraph on $n\ge1$
vertices $[n]=\{1,2,\ldots,n\}$ with $\Z_{\ge0}^k$-valued weights on
vertices. Explicitly, $X$ is defined by the tuple
$(\WW,\dd_1,\ldots,\dd_k)$ for a symmetric matrix
$\WW=(w_{ij})_{i,j=1}^n\in\Z_{\ge0}^{n\times n}$ and vectors
$\dd_\ell=(d_{\ell,1},\ldots,d_{\ell,n}) \in\Z_{\ge0}^n$, where
$w_{ij} =w_{ji}$ is number of edges between vertices
$i,j\in[n]$ and $(d_{1,i},\dots, d_{k,i})$ is the weight on vertex
$i\in[n]$. We also let $c(X)$ denote the number of connected
components of $X$.

For most of the paper, we take $k=1$, i.e. we have only one external
field, and suppress the subscript $\ell$ above, but setting things up
in this generality will help with the proofs later.

\begin{defn}\label{defn:leftconv}
    Let $G=(\AA,\hh_1,\ldots,\hh_k)$ be a weighted graph on vertex set
    $[N]$ and $X=(\WW,\dd_1,\ldots,\dd_k)$ be a test graph on vertex
    set $[n]$. We define the \emph{weighted homomorphism density}
    $t(X,G)$ as the quantity
    \begin{align}
    t(X,G):=\frac1{N^{c(X)}} \sum_{f:[n]\rightarrow[N]} \left(
    \prod_{i\in[n]} \prod_{\ell=1}^k h_{\ell,f(i)}^{d_{\ell,i}}
    \right) \left( \prod_{i\le j\in[n]} a_{f(i)f(j)}^{w_{ij}}
    \right)\, .
    \end{align}
    We say that a sequence $G_1,G_2,\ldots$ of weighted graphs is
    \emph{left-convergent} if for all test graphs $X$, the sequence
    $t(X,G_1),t(X,G_2),\ldots$ converges to a finite value.
\end{defn}

We remark that restricting to simple connected test graphs with zero
vertex weights recovers the notion of left convergence from
\cite{BCKL13}. In particular, our normalization factor $1/N^{c(X)}$
follows the \emph{sparse graph} regime considered in this paper,
though we will also allow our graphs to have unbounded (unweighted)
degree. Using this normalization, we have that $t(X_1\sqcup
X_2,G)=t(X_1,G)t(X_2,G)$ for all test graphs $X_1$ and $X_2$. Thus,
the linear span of $t(X,\cdot)$ over all test graphs $X$ forms an
algebra.

We will throughout assume we have a sequence of graphs $(G_N)_{N\ge1}$ with a
diverging number of vertices $N$, but often omit the subscript $N$ and
not mention the sequence explicitly.

As in \cite{BCLSV08}, it is also possible to consider \emph{injective
homomorphism densities}
\begin{equation}\label{eqn:inj}
    t_\inj(X,G) = \frac1{N^{c(X)}} \sum_{f:[n]\hookrightarrow[N]}
    \left( \prod_{i\in[n]} \prod_{\ell=1}^k h_{\ell,f(i)}^{d_{\ell,i}}
    \right) \left( \prod_{i\le j\in[n]} a_{f(i)f(j)}^{w_{ij}} \right)
\end{equation}
by replacing the sum in Definition~\ref{defn:leftconv} over all
functions $f:[n]\rightarrow[N]$ by a sum over all injections. However,
the notion of left-convergence is unchanged by a standard
inclusion-exclusion argument. Indeed, $t_\inj(X,G)$ can be
written as a finite linear combination of $t(X,G)$ and vice versa whenever $G$ is connected. We
will work with $t(X,G)$ for most of the paper but will switch to
$t_\inj(X,G)$ when convenient.

\begin{rem}\label{rmk:ConvEquivalence}
A closely related notion for sequences of sparse graphs is that of
\emph{local weak convergence}, first introduced in \cite{BS01} and
further developed in \cite{AS04} (for `geometric graphs') and
\cite{AL07,LRW23} (for `marked graphs'). As noted in \cite{BCKL13},
for unweighted graphs of uniformly bounded degree it is not hard to
see that the notions of left convergence and local weak convergence
are equivalent. In our setting of weighted graphs, it is also not hard
to see that left convergence is equivalent to local weak convergence
of marked graphs (as defined in \cite{AL07,LRW23}) when a uniform
bound on the vertex and edge weights and on degrees is also assumed.
 \end{rem}

Left convergence captures the local behavior of a sequence of graphs.
On the other hand, \emph{right convergence} of a sequence of graphs,
loosely defined as the convergence of homomorphism densities
\emph{into} test graphs rather than \emph{from}, captures its global
behavior by encoding the convergence of certain statistical physics
models. While previous works such as
\cite{BCLSV12,BCKL13,BCG17,BCCZ18} consider more general notions of
right convergence and their relation to left convergence, we will
restrict ourselves to the convergence of the free energy of Ising
models.

\begin{defn}\label{defn:ising}
    For a weighted graph $G=(\AA,\hh)$ on $[N]$, the \emph{Ising
    model} is the probability measure $\mu_G$ on $\{-1,1\}^N$
    assigning a mass of
    \begin{align}
    \mu_G(\bsigma)=\frac1{Z(G)} \exp\left( \frac12\<\bsigma,
    \AA\bsigma\> + \<\hh, \bsigma\> \right)\label{eq:ising_measure}
\end{align}
    to spin configurations $\bsigma\in\{-1,1\}^N$. Here,
    \begin{align}
    Z(G) = \sum_{\bsigma\in\{-1,1\}^N} \exp\left( \frac12\<\bsigma,
    \AA\bsigma\> + \<\hh, \bsigma\>
    \right)\label{eq:partition_function_def}
    \end{align}
    is the \emph{partition function} of the model. We furthermore
    define the \emph{normalized free energy} or \emph{free energy
    density} of the model as
    \begin{align*}
F(G) = \frac1N \log Z(G).
\end{align*}
\end{defn}

Our main result is that under certain conditions on the graph
sequence, left convergence implies convergence of normalized free
energies. Informally, for large graphs, the Ising free energy density
is a function of the local structure of the graph. Results of this
type were proved before in three settings:
\begin{enumerate}
\item[$(i)$] For ferromagnetic models (i.e. when
$A_{ij}\ge0$, $h_i\ge 0$ for all $i,j\in[N]$) on locally tree-like
graphs \cite{dembo2010ising,dembo2013factor}. Generalizations of these
results to Potts and random cluster models were obtained in
\cite{dembo2014replica} and \cite{bencs2023random}.
\item[$(ii)$] For general interactions under a
Dobrushin-type condition, which typically requires
$\|A\|_{\infty\to\infty}\le \delta_0$ for some small enough constant
$\delta_0>0$, see \cite{BCKL13}.
\item[$(iii)$] For sparse random graphs, under a
condition on the interactions that makes them amenable to spin glass
interpolation techniques first introduced by Guerra and Toninelli
\cite{guerra2002thermodynamic}. This approach was generalized and
applied to random graph sequences (among others) in
\cite{franz2003replica,bayati2010combinatorial,
abbe2014concentration,gamarnik2014right}.
\end{enumerate}
In summary ---excluding the `special' case of random graphs---
existing techniques require either the monotonicity afforded by
ferromagnetic interactions, or very weak interactions as guaranteed by the
Dorbrushin condition. In the case of random graphs, existing techniques make essential use of the ability to average over a whole ensemble of graphs and in particular cannot access certain deterministic pseudorandom graphs.

In contrast, Theorem \ref{thm:main} below establishes convergence of
Ising free energies for left-convergent graph sequences under a
spectral condition of the form $\lambda_{\max}(\AA),\lambda_{\min}(\AA)\in I$ for any interval $I$ of length less than $1$ which does not depend on $N$. Our
result additionally requires the graph sequence be sparse in the
sense of having $\|\AA\|_{\infty\to\infty}=O(1)$.

As illustrated by the next examples, the spectral condition covers a
significantly broader range examples than within earlier literature,
and is essentially required.
\begin{exm}[Spin glass models on locally
tree-like graphs]
\label{exm:spin-glass}
Let $G_N$ be a $k$-regular graph ($k\ge 3$) with girth $r(G_N)$ such
that $r(G_N)/(\log\log N)^2\to \infty$. Construct $\AA$ by letting
$A_{ij}\sim \Unif(\{-\beta,\beta\})$ independently for all $\{i,j\}\in
E(G_N)$, and $A_{ij}=0$ otherwise. Finally, take $\hh= h\bone_N$.

For this matrix $\AA$, \cite{mohanty2021explicit} (building on
\cite{bordenave2020new}) proves that $\|\AA\|_{2}\le
2\beta\sqrt{k-1}+o_
\P(1)$. Our main result implies that, if
$|\beta|<1/(4\sqrt{k-1})$, then the free energy density $F(G_N)$ converges
to the same finite limit as $N\to\infty$, irrespective of the graph
sequence.

In particular, the limit is the same as the free energy for a random
$k$-regular graph. A conjecture for the latter was derived by M\'ezard
and Parisi in \cite{mezard2001bethe} using the non-rigorous cavity
method from spin glass theory. Franz, Leone, and Toninelli
\cite{franz2003replica} proved an upper bound on the asymptotics that
matches the limit conjectured in \cite{mezard2001bethe}. Finally, for
the case $h=0$, $(k-1)(\tanh\beta)^2<1$, the limit free energy is
rigorously known.

We state the asymptotics for these graph sequences that follow from
our general theory as Corollary \ref{cor:spin-glass} below.
\end{exm}

\begin{exm}[Constrained ferromagnetic model
on locally tree-like graphs]
\label{exm:Constrained}
Consider a sequence of $k$-regular graphs satisfying the conditions of
Example \ref{exm:spin-glass}, but take now $A_{ij}=\beta>0$ if
$(i,j)\in E(G_N)$ and $A_{ij}=0$ otherwise, and also set $h=0$. However, rather than the
partition function $Z(G)$ of Eq.~\eqref{eq:partition_function_def}, we
consider the partition function for balanced configurations, given by
\begin{align}
    Z_{b}(G_N;\eps_N) = \sum_{\substack{\bsigma\in\{-1,1\}^N\\
    |\<\bsigma,\bone\>|\le N\eps_N}} e^{\<\bsigma, \AA_N\bsigma\>/2}\
    \label{eq:partition_function_def_0}
\end{align}
for some $\eps_N\in(0,1)$.
We will show that our main result allows us to characterize the
asymptotics of $F_b(G_N;\epsilon_N) = N^{-1}\log Z_b(G_N;\epsilon_N)$ for suitable sequences
$\eps_N\to 0$, provided that $\lambda_{\max}(\AA^{\perp})
$ and $\lambda_{\min}(\AA^{\perp})$ lie in a fixed interval of length $1-\delta$. Here,
$\AA^{\perp}=\PP^{\perp}\AA\PP^{\perp}$ is the projection of $\AA$
orthogonal to $\bone$ where we define $\PP^{\perp} = \II-\bone\bone^T/N$, and $\delta>0$ is an arbitrary constant. This
is stated as Corollary \ref{cor:Constrained} below.

In particular, if $G$ is a near-Ramanujan graph, this condition is
satisfied for $\beta<1/(4\sqrt{k-1})$.

Notice that this is a significantly broader range than what is covered by
a naive application of the spectral condition $\lambda_{\max}(\AA)
-\lambda_{\min}(\AA)\le 1-\delta$.
\end{exm}

\begin{exm}[Antiferromagnetic model on
locally tree-like graphs]
\label{exm:antiferromagnetic}
Finally under the same graph model as in the previous two examples,
take now $A_{ij}=-\beta \le 0$ if $\{i,j\}\in E(G_N)$ and $A_{ij}=0$
otherwise, and $\hh=\bzero$. Again we will obtain a characterization
of the asymptotics of $F(G_N)$ for $\lambda_{\max}(\AA^{\perp})
-\lambda_{\min}(\AA^{\perp})<1-\delta$, see Corollary
\ref{cor:Antiferromagnetic} below.
\end{exm}

\subsection*{Notation}

We collect here a few conventions used throughout the paper.
Vectors and matrices are written in boldface.  
The Euclidean inner product on $\R^N$ is denoted by
$\<\, \cdot\, ,\,\cdot\,\>$.

For a symmetric matrix $\MM\in\R^{N\times N}$, we denote
by $\lambda_1(\MM)\le \cdots\le \lambda_N(\MM)$ its eigenvalues.
We also write
$\lambda_{\max}(\MM)=\lambda_N(\MM)$ and $\lambda_{\min}(\MM)=\lambda_1(\MM)$
 for its largest and
smallest eigenvalues. The $\ell_2$ operator norm of a matrix
$\BB$ is denoted by $\|\BB\|$.

The collection of all probability measures on a measurable space $\Omega$ is denoted
by $\cuP(\Omega)$.

\section{Statement of main
results}\label{sec:main}

Let $\cG$ be a set of weighted graphs $G=(\AA,\hh)$ with one external
field. As mentioned, the functions $t(X,\,\cdot\,):\cG\rightarrow\R$
over all test graphs $X$ span an algebra of functions on $\cG$. We let
$\cL(\cG)$ denote the completion of this algebra with respect to the
uniform norm on $\cG$, which can be interpreted as the Banach algebra
of \emph{local functions} on $\cG$. By definition, $\{G_N:N\ge
1\}\subseteq\cG$ is a left-convergent sequence of graphs if and only
if, for each local function $f\in\cL(\cG)$, $\lim_{N\to\infty} f(G_N)$
exists and is finite.

We next state formally our assumptions on the graph sequence.
\begin{defn}\label{defn:conditions}
    Let $I\subset\R$ be a closed interval and $C>0$ be a constant. We
    define $\cG_{I,C}$ to be the set of finite weighted graphs
    $G=(\AA,\hh)$ satisfying the following conditions:
    \begin{enumerate}[label=(SW), ref=(SW)]
        \item\label{cond:SW}
        All eigenvalues of $\AA$ lie in $I$. If additionally the
        length of $I$ is strictly less than $1$, we refer to this as
        the \emph{spectral width condition}.
    \end{enumerate}
    \begin{enumerate}[label=(SG), ref=(SG)]
        \item\label{cond:SG}
        The operator norm bound $\|\AA\|_{\infty\rightarrow\infty}\le
        C$ holds, i.e. the $\ell^1$ norm of any row of $\AA$ is at
        most $C$. We refer to this as the \emph{sparse graph
        condition}.
    \end{enumerate}
    \begin{enumerate}[label=(BF), ref=(BF)]
        \item\label{cond:BF}
        The entrywise bound $\|\hh\|_\infty\le C$ holds, i.e. the
        magnitude of the external field $h_i$ on each spin $i\in[N]$
        is at most $C$. We refer to this as the \emph{bounded field
        condition}.
    \end{enumerate}

\end{defn}

Here is an `abstract' statement of our main result.

\begin{thm}\label{thm:main}
    Let $F$ be the free energy density function from
    Definition~\ref{defn:ising}. Then, for any constant $C>0$ and interval $I$ of length less than $1$, $F$ is a local function on $\cG_{I,C}$,
    i.e. $F\in\cL(\cG_{I,C})$.
\end{thm}

We also propose a more explicit formulation of the same statement.
%


\begin{cthm}{1'}[Explicit version of
Theorem \ref{thm:main}]
\label{thm:main_explicit}
Let $I$ be an interval of length less than $1$ and $\{G_N:N\ge 1\}$ be a sequence of graphs such that
$\lambda_{\max}(\AA_N),\lambda_{\min}(\AA_N)\in I$ and
$\|\AA_N\|_{\infty\rightarrow\infty},\|\hh_N\| _\infty\le C$ for each
$N\ge 1$.

If the weighted homomorphism densities $t(X,G_N)$ converge to finite
limits $t_{\infty}(X)$ for all test graphs $X$, then the free energy
density $F(G_N)$ converges to a finite limit $F_{\infty}$ depending
uniquely on the limits $\{t_{\infty}(X)\}$.
\end{cthm}

Let $\overline{\cG_{I,C}}$ denote the completion of $\cG_{I,C}$ under
the weak topology induced by $\cL(\cG_{I,C})$. Then
$\overline{\cG_{I,C}}$ is a subset of unimodular random rooted graphs
which is compact under the weak topology, for example by arguments
from \cite{AL07}. Thus, an equivalent way to state
Theorem~\ref{thm:main} is that the free energy density function $F$
extends to a continuous function on $\overline{\cG_{I,C}}$ under the
weak topology, for each interval $I$ of length at most $1-\delta$.

 \begin{rem}\label{rmk:necessity}
 The spectral condition is tight up to a factor of $2$, in the sense specified by
 Proposition \ref{prop:NecessitySpectral} below.

On the other hand, we believe the sparsity condition is not necessary; see the discussion in Section~\ref{sec:future}.
 \end{rem}

\section{Applications}\label{sec:remarks}
\label{sec:Disc}
\label{sec:coro}

In this section we show how Theorem \ref{thm:main} yields concrete
predictions for the three examples of Section \ref{sec:statement}.
Next, we discuss the necessity of the spectral assumption by providing
the construction mentioned in Remark \ref{rmk:necessity}. All proofs
from this section are deferred to Appendix \ref{app:CoroProofs}.

\begin{exmcont}{1}  We state formally
the asymptotics of the free energy of spin glasses on locally
tree-like graphs.
\begin{cor}\label{cor:spin-glass}
Let $G_N=(\AA_N,\hh_N=\bzero)$ be a sequence of weighted $k$-regular
graphs ($k\ge 3$) with girth $r(G_N)/(\log\log N)^2\to \infty$, and
weights $A_{ij}\sim \Unif(\{-\beta,\beta\})$. If
$|\beta|<1/(4\sqrt{k-1})$, then
    \begin{align}\label{eq:SpinGlassFreeEnergy}
    \lim_{N\to\infty} F(G_N) = \log 2+\frac{k}{2}\log \cosh(\beta) .
    \end{align}
    \end{cor}
As mentioned in the introduction, the proof of this result follows
from Theorem \ref{thm:main} which implies that the free energy density
is a local function, under the stated assumptions. We then use known
results about the free energy density of random regular graphs to
obtain the explicit expression of Eq.~\eqref{eq:SpinGlassFreeEnergy}.

We emphasize here that the role of the uniform random signing is only to ensure that the spectral condition holds and that the local weak limit is the uniform random signing of the $k$-regular tree. In particular, the same result holds for any deterministic sequence of signed $k$-regular graphs (converging locally weakly to the i.i.d. uniform random signing of the $k$-regular tree) satisfying the spectral condition.
\end{exmcont}

\begin{exmcont}{2}
Our next result characterizes the free energy density of magnetization-constrained ferromagnetic models.
    \begin{cor}\label{cor:Constrained}
    Let $G_N=(\AA_N,\hh_N=\bzero)$ be a sequence of $k$-regular graphs
    ($k\ge 3$), converging locally to the $k$-regular tree, with edge
    weights $A_{ij}=\beta\ge 0$ for all $\{i,j\}\in E(G_N)$. Recall that
    $\AA^{\perp}_N= \PP^{\perp}\AA_N\PP^{\perp}$, and define
    $F_b(G_N;\eps_N) = N^{-1}\log Z_b(G_N;\eps_N)$, with the
    constrained partition function $Z_b(G_N;\eps_N)$ as per
    Eq.~\eqref{eq:partition_function_def_0}. If
    $\lambda_{\max}(\AA_N^{\perp}),\lambda_{\min} (\AA_N^{\perp})\in I$, for some $N$-independent interval $I$ of length less than $1$, then
    there exists a sequence $\eps_N\to 0$ such that
        \begin{align}
        \lim_{N\to\infty} F_b(G_N;\eps_N) = \log 2+\frac{k}{2}\log
        \cosh(\beta) .
        \end{align}
        \end{cor}
        \end{exmcont}

    \begin{exmcont}{3}
Finally, we consider the free energy density of antiferromagnetic
models.
    \begin{cor}\label{cor:Antiferromagnetic}
    Let $G_N=(\AA_N,\hh_N=\bzero)$ be a sequence of $k$-regular graphs
    ($k\ge 3$), converging locally to the $k$-regular tree, with edge
    weights $A_{ij}=\beta\le 0$ for all $\{i,j\}\in E(G_N)$. Recall that
    $\AA^{\perp}_N= \PP^{\perp}\AA_N\PP^{\perp}$. If
    $\lambda_{\max}(\AA_N^{\perp}),\lambda_{\min} (\AA_N^{\perp})\in I$, for some $N$-independent interval $I$ of length less than $1$, then
        \begin{align}
        \lim_{N\to\infty} F(G_N) = \log 2+\frac{k}{2}\log \cosh(\beta)
        .
        \end{align}
        \end{cor}
    \end{exmcont}

We finally give an example that elucidates the tightness up to constants of the
spectral assumption.
\begin{prop}\label{prop:NecessitySpectral}
For any fixed $\delta>0$, we can construct two graph sequences
$G_N^{(1)}=(\AA^{(1)}_N,\hh^{(1)}_N)$ and
$G_N^{(2)}=(\AA^{(2)}_N,\hh^{(2)}_N)$ with the same left limit such that
$\lambda_{\max}(\AA^{(\ell))}_N)$ and $\lambda_{\min} (\AA^{(\ell)}_N)$ lie in an $N$-independent interval of length $2+\delta$, 
but they have different limiting free energy
densities.
\end{prop}
The construction used to prove this proposition is well known in the
literature. It amounts to considering an antiferromagnetic Ising
model, either on random regular graphs or on \emph{bipartite} random
regular graphs. We nevertheless provide a detailed proof for
completeness.




\section{Proof of
Theorem~\ref{thm:main}}\label{sec:mainproof}

In this section we present the proof of Theorem~\ref{thm:main},
deferring some technical results to Appendix~\ref{app:lemmas}. We will begin by
providing an outline of the proof.

\subsection{Proof outline}

The starting point is to transform the Ising model into a continuous
Gibbs measure by a technique from \cite{baker1962certain}. This
transformation expresses the Ising measure \eqref{eq:ising_measure} as
a mixture of product measures on $\{-1,1\}^N$. This identity allows us
to write the free energy of the Ising model as the sum of the free
energy of a certain continuous Gibbs measure on $\R^N$ and a
log-determinant term which is easily shown to lie in $\cL(\cG_{I,C})$.
This part of the proof is carried out in
Section~\ref{subs:renormalize}.

By reducing the problem to computing the free energy of a continuous
model, we are able to leverage analytic tools that are not available
in the discrete setting. As first noted in \cite{BB19} (see also
\cite{BBD24} for a comprehensive survey), this continuous Gibbs
measure is strongly log-concave whenever the original Ising
model satisfies the spectral width condition \ref{cond:SW}. This in
turn implies rapid mixing of Langevin dynamics for this Gibbs measure
via Bakry-\'Emery theory \cite{bakry2014analysis}.

We next use an interpolation argument to express the free energy as
the integral of the average energy over a certain range of inverse
temperature parameters, and the rapid mixing of Langevin dynamics
allows us to approximate this by instead integrating the average
energy achieved by Langevin dynamics after a large constant amount of
time. The log-concavity of the measure is essential to making this
argument work. This part of the proof is carried out in
Sections \ref{subs:langevin} and \ref{sec:UnifConv}.

It then remains to show that the average energy of the renormalized
model achieved by Langevin dynamics after a constant amount of time
lies in $\cL(\cG_{I,C})$. We achieve this via a discretization
argument, where we show first that the Euler-Maruyama scheme computes
this average energy within an arbitrarily small error by adjusting the
step size. After this, we further approximate the output of the
Euler-Maruyama scheme by using polynomial approximations to the
Langevin drift term, which is the gradient of the renormalized
potential. The conditions \ref{cond:SG} and \ref{cond:BF} are crucial
here to ensure that these polynomial approximations can be made
uniform in the system size $N$ through a uniform bound on the
sub-Gaussian norms of each coordinate throughout the discretized
dynamics. Finally, we rewrite this polynomial expression in terms of
weighted homomorphism densities $t(X,G)$, after which the proof is
complete. This part of the proof is carried out in
Section~\ref{subs:discretize}.

\subsection{Reduction to a log-concave Gibbs measure}\label{subs:renormalize}

The goal of this subsection is to reduce proving
Theorem~\ref{thm:main} to proving Theorem~\ref{thm:contloc} (stated
below), which is an analogous result for a continuous variant of the
model with real-valued spins rather than $\{-1,1\}$-valued spins.

Throughout, we will assume that $G=(\AA,\hh)\in\cG_{I,C}$ is a
weighted graph on $[N]$ for a fixed interval $I$ with $|I|\le 1-\delta$. Let $\II\in\R^{N\times N}$ be the identity
matrix. Note that by definition, $F(\AA+t\II,\hh)=F(\AA,\hh)+\frac
t2$. As constant functions are in $\cL(\cG_{I,C})$ by considering
$t(X,\cdot)$ for $X$ the test graph on one vertex with zero vertex and
edge weights, we may shift $I$ and without loss of generality assume
that $I=[-1+\lambda,-\lambda]$ for some $\lambda>0$.

For clarity, we will denote $\AA'=-\AA$ so that the spectrum of $\AA'$
is contained in $[\lambda,1-\lambda]$. Following \cite{BB19}, we write
$(\AA')^{-1}=(c\II)^{-1}+\BB^{-1}$ for $c=c_\lambda=1-\frac\lambda2$,
which means that $\BB$ is positive definite with spectrum contained in
$[\lambda,2\lambda^{-1}]$. Since the Gaussian measure with covariance
$(\AA')^{-1}$ is the convolution of the Gaussian measures with covariances
$\II/c$ and $\BB^{-1}$, we may write
\begin{align*}
\exp\left(-\frac12\<\bsigma, \AA'\bsigma\>\right) =
\frac{\det(c\II)^{1/2}
\det(\BB)^{1/2}}{\det(2\pi\AA')^{1/2}}  \int_{\R^N} \exp\left(-\frac c2
\|\bsigma-\bphi\|^2\right) \exp\left(-\frac12\<\bphi,
\BB\bphi\>\right)\de\bphi,
\end{align*}
which means that
\begin{align*}
\sum_{\bsigma\in\{-1,1\}^N} \exp\left(-\frac12\<\bsigma, \AA'\bsigma\>
+\<\hh, \bsigma\>\right)
\end{align*}
\begin{align*}
=\frac{\det(c\II)^{1/2}
\det(\BB)^{1/2}}{\det(2\pi\AA')^{1/2}}
\int_{\R^N}\left(\sum_{\bsigma\in\{-1,1\}^N} \exp\left(-\frac
c2\|\bsigma-\bphi\|^2 + \<\hh, \bsigma\>\right)\right)
\exp\left(-\frac12\<\bphi, \BB\bphi\>\right)\de\bphi.
\end{align*}

Let $\Tr(\MM)$ and $\tr(\MM)=\tr_N(\MM)=\frac1N\Tr(\MM)$ denote the
trace and normalized trace of an $N\times N$ matrix $\MM$. Defining
the $h$-renormalized potential $V_h:\R\to\R$ by
\begin{align}\label{eqn:Vh}
V_h(\varphi) &:= -\log\left(\exp\left(-\frac c2(\varphi-1)^2+h\right) +
\exp\left(-\frac c2(\varphi+1)^2-h\right)\right) \\
&= \frac
c2(\varphi^2+1) - \log(2\cosh(c\varphi+h))\nonumber
\end{align}
and the $\hh$-renormalized potential $V_\hh:\R^N\to\R$ by
\begin{align*}
V_\hh(\bphi) &:= 
\sum_{i=1}^NV_{h_i}(\varphi_i)\, ,
\end{align*}
 the free energy can be written as
\begin{align}
    F(-\AA',\hh) &= \frac1N\log\frac{\det(c\II)^{1/2}
\det(\BB)^{1/2}}{\det(2\pi\AA')^{1/2}}
    \exp\left(-\frac12\<\bphi, \BB\bphi\>\right)
    e^{-V_\hh(\bphi)}\de\bphi. \label{eqn:FAh}\\ &= -\frac12\tr(2\pi\log\AA')
    + \frac12\tr(\log\BB) + \frac12\log(c) +
    \frac1N\log\int_{\R^N} \exp\left(-\frac12\<\bphi,
    \BB\bphi\>\right) e^{-V_\hh(\bphi)}\de\bphi. \nonumber
\end{align}

Here, $\log\AA'$ and $\log\BB$, and in general $g(\MM)$ for continuous
functions $g:\R\rightarrow\R$ and symmetric matrices $\MM$, are
defined using the usual functional calculus for self-adjoint
operators. The following proposition shows that the first three terms
are in $\cL(\cG_{I,C})$.

\begin{prop}\label{prop:trace}
    For all continuous functions $g:I\rightarrow\R$,
    $\tr(g(\AA))\in\cL(\cG_{I,C})$, i.e. the function on $\cG_{I,C}$
    mapping $G=(\AA,\hh)$ to $\tr(g(\AA))$ is local.
\end{prop}

\begin{proof}
    Note that if $X$ is the cycle on $\ell$ vertices with zero vertex
    weights, then
    \begin{align*}
t(X,G) = \frac1N \sum_{f:\Z/\ell\Z\rightarrow[N]} \prod_{i \in
\Z/\ell\Z} a_{f(i)f(i+1)} = \tr(\AA^\ell).
\end{align*}
    It follows that $\tr(p(\AA))\in\cL(\cG_{I,C})$ for all polynomials
    $p$. On the other hand, the Stone-Weierstrass theorem implies that
    the $\cG_{I,C}$-uniform closure of $\tr(p(\AA))$ over polynomials
    $p$ is given by $\tr(g(\AA))$ over continuous functions
    $g:I\rightarrow\R$, which implies the proposition.
\end{proof}

Indeed, $\AA\mapsto\log\AA'$ and $\AA\mapsto\log\BB$ correspond to the
continuous functions $t\mapsto\log(-t)$ and
$t\mapsto\log(-\frac{ct}{t+c})$ on $[-1+\lambda,-\lambda]$. Thus, to
show that $F\in\cL(\cG_{I,C})$, it remains to study the last term
\begin{equation}\label{eqn:Benergy}
\frac1N\log\int_{\R^N} \exp\left(-\frac12\<\bphi, \BB\bphi\>\right)
e^{-V_\hh(\bphi)}\de\bphi.
\end{equation}
This term can be interpreted as the free energy density of a Gibbs
measure on $\R^N$ with base measure $e^{-V_\hh(\bphi)}\de\bphi$ and
quadratic Hamiltonian $-\frac12\<\bphi, \BB\bphi\>$. Our goal now is
to show that the free energy density of this continuous model is a
local function of $\AA$. However, it is more natural to consider it as
a function of $\BB$. In order to do so, we first show that, very
loosely speaking, $\BB$ is a continuous function of $\AA$ under the
local topology so local functions of $\BB$ are also local functions of
$\AA$.

\begin{lem}\label{lem:polyhom}
    Let $p\in\R[x]$ be a polynomial. For all test graphs $X$, there
    exist test graphs $X_1,\ldots,X_r$ and constants
    $c_1,\ldots,c_r\in\R$ (depending on $p$ and $X$)
    such that for all weighted graphs
    $G=(\AA,\hh)$, we have that
    \begin{align*}
t(X,(p(\AA),\hh))=\sum_{i=1}^rc_it(X_i,G).
\end{align*}
    In other words, weighted homomorphism densities into the graph
    defined by a polynomial of the adjacency matrix of $G$ can be
    written as \emph{fixed} linear combinations of other weighted
    homomorphism densities into $G$. Consequently, local functions of
    $(p(\AA),\hh)$ are also local functions of $G$.
\end{lem}

\begin{proof}
    Although the notation is a bit cumbersome, the proof idea is simply to
    expand the left-hand side and take $X_i$ to be given by blowing up edges of
    $X$ with paths corresponding to degrees of monomials in $p$. Let
    $X=(\WW,\dd)$ for $\WW=(w_{ij})_{i,j=1}^n\in\Z_{\ge0}^{n\times n}$
    and $\dd\in\Z_{\ge0}^n$.

    Since $t$ is multiplicative in $X$ under disjoint unions, it
    suffices to prove the lemma when $X$ is connected. In this case,
    let $E=E(X)$ be the multiset of edges in $X$ and
    $p(x)=\sum_{\ell=0}^m c_\ell x^\ell$. For
    $(\ell_e)\in\{0,1,\ldots,m\}^E$, let $c_{(\ell_e)}=\prod_{e\in
    E}c_{\ell_e}$ and let $X_{(\ell_e)}$ denote the test graph
    resulting from the following procedure:
    \begin{enumerate}
        \item Replace each edge $e\in E$ of $X$
        with a path of length $\ell_e$ if $\ell_e>1$, adding new
        vertices to $X$ for internal vertices. These new internal
        vertices $v$ not originally in $X$ are assigned a weight $d_v$
        of $0$. If $e$ is a self-loop then this creates a cycle of length $\ell_e$ passing through the vertex of $e$.
        \item Contract each edge $e\in E$ of $X$
        with $\ell_e=0$ and identify the endpoints $e_1$ and $e_2$ of
        $e$ as a single vertex $v$. The weight $d_v$ of this
        contracted vertex $v$ is the sum $d_{e_1}+d_{e_2}$ of the
        weights of the endpoints of $e$ before the contraction if $e_1\ne e_2$ and is $d_{e_1}$ otherwise, i.e. we sum the weights unless the edge is a self-loop in which case it is just deleted.
    \end{enumerate}

    It is easy to see that the order in which the blowups and
    contractions of edges are performed does not matter and that the
    result of this procedure is a well-defined test graph.
     By our definitions, we then have the
    following chain of equalities:
    \begin{align*}
        t(X,(p(\AA),\hh)) &= \frac1N \sum_{f:[n]\rightarrow[N]} \left(
        \prod_{i\in[n]} \prod_{\ell=1}^k h_{\ell,f(i)}^{d_{\ell,i}}
        \right) \left( \prod_{i\le
        j\in[n]}(p(\AA))_{f(i)f(j)}^{w_{ij}} \right) \\ &= \frac1N
        \sum_{f:[n]\rightarrow[N]} \left( \prod_{i\in[n]}
        \prod_{\ell=1}^k h_{\ell,f(i)}^{d_{\ell,i}} \right) \left(
        \prod_{e = \{e_1,e_2\} \in E}(p(\AA))_{f(e_1)f(e_2)} \right)
        \\ &= \frac1N \sum_{f:[n]\rightarrow[N]} \left(
        \prod_{i\in[n]} \prod_{\ell=1}^k h_{\ell,f(i)}^{d_{\ell,i}}
        \right) \left( \sum_{(\ell_e)\in\{0,1,\ldots,m\}^E}
        \prod_{e\in E} c_{\ell_e}(\AA^{\ell_e})_{f(e_1)f(e_2)} \right)
        \\ &= \sum_{(\ell_e)\in\{0,1,\ldots,m\}^{E}} c_{(\ell_e)}
        \left( \frac1N \sum_{f:[n]\rightarrow[N]} \left(
        \prod_{i\in[n]} \prod_{\ell=1}^k h_{\ell,f(i)}^{d_{\ell,i}}
        \right) \left( \prod_{e\in E}(\AA^{\ell_e})_{f(e_1)f(e_2)}
        \right) \right) \\ &= \sum_{(\ell_e)\in\{0,1,\ldots,m\}^{E}}
        c_{(\ell_e)} t(X_{(\ell_e)},G).
    \end{align*}

    This is indeed an expression of $t(X,(p(\AA),\hh))$ as a linear
    combination of $t(X_i,G)$ that does not depend on $G$. It follows
    that if $f$ is a local function of $(p(\AA),\hh)$, then
    $f(p(\AA),\hh)$ is a local function of $G$. Formally, by this we
    mean that if $I'\subset\R$ is a closed interval containing the
    image of $I$ under $p$ and
    $C'\ge\max(C,\sum_{\ell=0}^m|c_\ell|C^\ell)$, then for all local
    functions $f\in \cL(\cG_{I',C'})$, the map $(\AA,\hh)\mapsto
    f(p(\AA),\hh)$ is in $\cL(\cG_{I,C})$. Indeed, any uniform
    approximation of $f$ by linear combinations of $t(X,(p(\AA),\hh))$
    can be written as linear combinations of $t(X,G)$ by what we have
    just shown.
\end{proof}

Now, we will show that \eqref{eqn:Benergy} is a local function of
$(\BB,\hh)$ on an appropriate set of graphs $(\BB,\hh)$. This does not
immediately imply Theorem~\ref{thm:main} by Lemma~\ref{lem:polyhom}
since $\BB=-\frac{c\AA}{\AA+c\II}$ is not a polynomial function of
$\AA$. However, we will be able to apply Lemma~\ref{lem:polyhom} to
arbitrary truncations of the Neumann series
\begin{equation}\label{eqn:neumann}
\BB=\sum_{\ell\ge1}-(-c)^{-\ell+1}\AA^\ell
\end{equation}
and use a Lipschitz property of the free energy density
\eqref{eqn:Benergy} in $\BB$ to conclude the desired result. In order
to carry out this argument, we first rewrite things in terms of a more
general setup.

\begin{defn}\label{defn:contFV}
    For $\alpha>0$, an \emph{$\alpha$-good potential} is a
    function $V:\R\times \R\to\R$,  $(h,\varphi)\mapsto V_h(\varphi)$ 
    which is twice continuously differentiable 
    with 
    \[
        \sup_{\substack{|h|\le C \\ \varphi\in\R}}
        \left| \frac\partial{\partial h} V_h'(\varphi)\right| < \infty
    \]
    for all $C>0$ and $V_h''(\varphi)\in[\alpha,\alpha^{-1}]$
    for all $h,\varphi\in\R$ (where  $V_h'$ and $V_h''$ are the first and
     second  derivatives of $V_h$ with respect to $\varphi$). In other words, this family of potential
    functions is smooth and uniformly strongly convex with Lipschitz
    gradient, and furthermore $V_h'$ is Lipschitz in $h$. For
    $\hh\in\R^N$, we define the potential
    \begin{align*}
V_\hh(\bphi):=\sum_{i=1}^NV_{h_i}(\varphi_i)\, .
\end{align*}

    Given an $\alpha$-good potential $V$ and
    a weighted graph $G=(\BB,\hh)$ on $[N]$ with $\BB\succ-\alpha\II$,
    we consider the Gibbs measure on $\R^N$ with base measure
    $e^{-V_\hh(\bphi)}\de\bphi$ and quadratic Hamiltonian
    $-\frac12\<\bphi, \BB\bphi\>$ and denote its free energy density
    by
    \begin{align*}
F_V(G)=F_V(\BB,\hh)=\frac1N \log \int_{\R^N} \exp\left( -\frac12
\<\bphi, \BB\bphi\> \right) e^{-V_\hh(\bphi)} \de\bphi.
\end{align*}
\end{defn}

Note that if $V$ is $\alpha$-good, then $F_V(\BB,\hh)$ is indeed
well-defined for all $\BB$ with $\BB\succ-\alpha\II$ by the
$\alpha$-strong convexity of $V_\hh$ as the Gibbs measure is then
strongly log-concave. Also, the renormalized potential \eqref{eqn:Vh}
from the Ising model satisfies
\begin{align*}
    1>c\ge V_h''(\varphi) &= c-\frac{c^2}{\cosh^2(c\varphi+h)}\ge
    c-c^2>0 \\ \left| \frac\partial{\partial h}V_h'(\varphi) \right|
    &= \left| -\frac c{\cosh^2(c\varphi+h)} \right| \le c
\end{align*}
and is clearly smooth in $h$ and $\varphi$ so it is a $(c-c^2)$-good
 potential. We will henceforth work in the generality of
$\alpha$-good potentials.

In Sections~\ref{subs:langevin} and~\ref{subs:discretize}, we will prove the following result,
which can be viewed as the continuous analogue of
Theorem~\ref{thm:main} in a strongly convex potential.

\begin{thm}\label{thm:contloc}
    For any $\alpha$-good potential $V$ and $C>0$, we have
    that $F_V\in\cL(\cG_{[0,C],C})$, where $F_V$ is the free energy
    density function from Definition~\ref{defn:contFV}.
\end{thm}

In other words, the renormalized free energy density is a local
function of the (positive semi-definite) coupling matrix $\BB$ and external
field $\hh$, given a uniform bound on the maximum eigenvalue of $\BB$ 
and on  $\|\hh\|_\infty$. To complete the proof of
Theorem~\ref{thm:main}, we need the following lemma to show that we
can pass to the limit in the Neumann series \eqref{eqn:neumann}. Its
proof relies on a uniform bound on the second moment of certain Gibbs
measures from Lemma~\ref{lem:secondmoment} and is deferred to
Section~\ref{sec:UnifConv}.

\begin{lem}\label{lem:FVlip}
    For any $\alpha$-good potential $V$ and $C>0$, there
    exists a constant $L>0$ depending only on $V$ and $C$ such that
    for any positive semidefinite $N\times N$ matrices $\BB_1$ and
    $\BB_2$ with spectrum contained in $[0,C]$ and vectors
    $\hh\in\R^N$ with $\|\hh\|_\infty\le C$, we have that
    \begin{align*}
|F_V(\BB_1,\hh)-F_V(\BB_2,\hh)|\le L\|\BB_1-\BB_2\|.
\end{align*}
   In other words, $F_V(\BB,\hh)$ is a
    Lipschitz function of $\BB\in\{\MM:\bzero\preceq\MM\preceq C\II\}$
    under the  operator norm. 
\end{lem}

\begin{proof}[Proof of Theorem~\ref{thm:main}]

    Recall that the spectrum of $\AA$ is contained in
    $[-1+\lambda,-\lambda]$ and $c=1-\frac\lambda2$, so
    $\frac{\|\AA\|}c\le\frac{2-2\lambda} {2-\lambda}<1$ and the
    Neumann series
    \begin{align*}
\BB=\sum_{\ell\ge1} -(-c)^{-\ell+1} \AA^\ell
\end{align*}
    converges in operator norm. Defining
    \begin{align*}
\BB_k=\sum_{\ell=1}^k -(-c)^{-\ell+1} \AA^\ell,
\end{align*}
    we have the sparsity bound
    \begin{align*}
\|\BB_k\|_{\infty\rightarrow\infty} \le \sum_{\ell=1}^k c^{-\ell+1}
\|\AA\|_{\infty\rightarrow\infty}^\ell \le C_k
\end{align*}
    for $C_k=\sum_{\ell=1}^k c^{-\ell+1}C^\ell$.

    By Theorem~\ref{thm:contloc}, $F_V\in\cL(\cG_{[0,C'],C'})$ for
    $C'=\max(C,C_k)$, so applying Lemma~\ref{lem:polyhom} to the
    polynomial $p(x)=\sum_{\ell=1}^k(-1)^\ell c^{-\ell+1}x^\ell$ gives us that
    the function given by $(\AA,\hh)\mapsto F_V(\BB_k,\hh)$ is in
    $\cL(\cG_{I,C})$. On the other hand, we have by
    Lemma~\ref{lem:FVlip} that there exists a constant $L$ depending
    only on $V$ and $C''=\max(C,c\sum_{\ell\ge1}(\frac{2-2\lambda}
    {2-\lambda})^\ell)$ such that
    \begin{align*}
|F_V(\BB,\hh)-F_V(\BB_k,\hh)|\le L\|\BB-\BB_k\| \le cL\sum_{\ell>k}
\left(\frac{2-2\lambda}{2-\lambda}\right) ^\ell.
\end{align*}
    This bound is uniform in $G=(\AA,\hh)\in\cG_{I,C}$ and tends to
    $0$ as $k\rightarrow\infty$, so as functions of $G=(\AA,\hh)$, we
    have that $F_V(\BB,\hh)=\lim_{k\rightarrow\infty} F_V(\BB_k,\hh)$
    in $\cL(\cG_{I,C})$.

    This concludes the proof. Indeed, recall from \eqref{eqn:FAh} that
    we have the identity
    \begin{align*}
F(\AA,\hh) = -\frac12\tr(\log(-2\pi\AA)) + \frac12\tr(\log\BB) +
\frac12\log(c) + F_V(\BB,\hh),
\end{align*}
    and each of these terms is in $\cL(\cG_{I,C})$ by
    Proposition~\ref{prop:trace} and what we have just argued.
\end{proof}

It thus remains to prove Theorem~\ref{thm:contloc} and
Lemma~\ref{lem:FVlip}, which occupy the remainder of this section.

\subsection{Interpolation and Langevin
dynamics}\label{subs:langevin}

Our goal now is to prove Theorem~\ref{thm:contloc}. In this
subsection, we will reduce this to proving
Proposition~\ref{prop:unifconv} and Proposition~\ref{prop:locenergy}.
We will then prove Proposition~\ref{prop:unifconv} in  Section \ref{sec:UnifConv}
and Proposition~\ref{prop:locenergy} in Section \ref{subs:discretize}.

For now, we will fix $\hh$ and $\BB$ up to a constant factor and
consider the function $F_G:[0,\infty)\rightarrow\R$ given by
$F_G(\beta):=F_V(\beta\BB,\hh)$, where $G=(\BB,\hh)$. Here, $V$ is a
fixed $\alpha$-good family of potentials and is suppressed from the
notation. We will interpret $F_G(\beta)$ as the free energy of a Gibbs
measure $\mu_\beta$ and expectations with respect to $\mu_\beta$ will
be denoted by $\langle\, \cdot\,\rangle_\beta$. We use $\bphin$
to denote a random vector with distribution $\bphin\sim\mu_\beta$. Hence
\begin{align*}
\langle f(\bphin)\rangle_\beta & = \int_{\R^N} f(\bphi)\, \mu_\beta(\de\bphi)\\
& =
e^{-NF_G(\beta)} \int_{\R^N}f(\bphi) \exp\left(-\frac12\<\bphi,
\beta\BB\bphi\>\right) e^{-V_\hh(\bphi)} 
\de\bphi\, .
\end{align*}
for bounded measurable functions $f$. 
We can interpret  $\beta$ as an inverse temperature parameter. We also let
$\Var_\beta(\,\cdot\,)$, $\Cov_\beta(\,\cdot\,)$ denote the variance and
covariance with respect to $\mu_\beta$. For convenience, we denote the
Hamiltonian of this system by $H(\bphi)=\frac12\<\bphi, \BB\bphi\>$.

Recall the standard statistical physics identity
$F_G'(\beta)=-\frac1N\langle H(\bphin)\rangle_\beta$, which gives
rise to the interpolation scheme
\begin{align*}
F_V(G) = F_G(1) = F_G(0) + \int_0^1 F_G'(\beta)\de\beta =
F_V(\bzero,\hh) - \int_0^1 \frac1N\langle H(\bphin)\rangle_\beta
\de\beta.
\end{align*}
Our approach will be to show that $F_V(\bzero,\hh)$ and
$\frac1N\langle H(\bphin)\rangle_\beta$ are in $\cL(\cG_{[0,C],C})$.
First, we will dispense with the boundary term $F_G(0)$.

\begin{prop}\label{prop:FV0h}
    The function on $\cG_{[0,C],C}$ mapping $G=(\BB,\hh)$ to
    $F_V(\bzero,\hh)$ is in $\cL(\cG_{[0,C],C})$.
\end{prop}

\begin{proof}
    We have that
    \begin{align*}
F_V(\bzero,\hh) = \frac1N \sum_{i=1}^N \log \int_\R
e^{-V_{h_i}(\varphi)}\de\varphi = \frac1N \sum_{i=1}^N\log Z(h_i)
\end{align*}
    for $Z(h)=\int_\R e^{-V_h(\varphi)}\de\varphi$. Since $V_h(\varphi)$
    is smooth in $h$ and is $\alpha$-strongly convex, $Z(h)$ is a
    continuous function on $[-C,C]$ by the dominated convergence
    theorem. By considering $t(X,G)=\frac1N\sum_{i=1}^Nh_i^k$ for $X$
    the test graph consisting of one vertex of weight $k$, we have
    that $\frac1N\sum_{i=1}^Np(h_i)$ is in $\cL(\cG_{[0,C],C})$ for
    all polynomials $p\in\R[x]$. The result then follows by the
    Stone-Weierstrass theorem, since we can approximate $\log Z$ by
    polynomials on $[-C,C]$.
\end{proof}

Our goal is now to prove that the main term $\int_0^1\frac1N\langle
H(\bphin)\rangle_\beta \de\beta$ is also in $\cL(\cG_{[0,C],C})$.

Before proceeding to the actual proof, let us outline our strategy. 
We will prove two facts: 
\begin{enumerate}
\item[{\sf S1.}] $\langle
H(\bphin)\rangle_\beta/N$ is in $\cL(\cG_{[0,C],C})$ for each $\beta$;
\item[{\sf S2.}] $\beta\mapsto \frac1N\langle H(\bphin)\rangle_\beta$ is
a norm-continuous map from $[0,1]$ to $\cL(\cG_{[0,C],C})$
(where $\cL(\cG_{[0,C],C})$ is endowed with the uniform norm). 
\end{enumerate}
By the second property, the integration may be carried out in
$\cL(\cG_{[0,C],C})$ as a Bochner integral, and hence the integral is in $\cL(\cG_{[0,C],C})$.

To
prove  norm-continuity, we will first find approximations
$\langle\,\cdot\,\rangle^{(t)}_\beta$ of the Gibbs measures
$\langle\,\cdot\,\rangle_\beta$ such that
\begin{equation}\label{eqn:(t)conv}
    \frac1N \langle H(\bphin)\rangle_\beta^{(t)}
    \quad\stackrel{t\rightarrow\infty} \longrightarrow\quad \frac1N
    \langle H(\bphin)\rangle_\beta
\end{equation}
uniformly over $(\beta,G)\in[0,1]\times\cG_{[0,C],C}$. Then, it
suffices to prove that $\beta\mapsto \frac1N\langle
H(\bphin)\rangle_\beta^{(t)}$ is a norm-continuous map from $[0,1]$
to $\cL(\cG_{[0,C],C})$ for each $t$. This will be accomplished by
constructing uniform approximations of the path
$\beta\mapsto\frac1N\langle H(\bphin)\rangle_\beta^{(t)}$ given by
polynomials in $\beta$ with coefficients that are linear combinations
of $t(X,G)$ for suitable test graphs $X$.
%

We now proceed with the actual proof. We consider Langevin dynamics with respect to $\mu_\beta$,
defined by the stochastic differential equation (SDE)
\begin{equation}\label{eqn:langevin}
    \de\bphin^{(t)} = -\beta\grad H(\bphin^{(t)})\,\de t -\grad
    V_\hh(\bphin^{(t)}) \de t +\sqrt{2} \de\ww^{(t)}
\end{equation}
for $\ww^{(t)}$ a standard Brownian motion on $\R^N$. This diffusion
has invariant measure $\mu_\beta$. Furthermore, for an initial
condition $\bphin^{(0)}$ with law $\mu_\beta^{(0)}$, we have that
\begin{equation}\label{eqn:KLcontract}
    D(\mu_\beta^{(t)}\|\mu_\beta)\le e^{-\alpha
    t}D(\mu_\beta^{(0)}\|\mu_\beta),
\end{equation}
where $\mu_\beta^{(t)}$ denotes the law of $\bphin^{(t)}$ and
$D(\,\cdot\,\|\,\cdot\,)$ denotes Kullback-Leibler (KL) divergence
(a.k.a. relative entropy). 
This holds by \cite[Theorem 5.2.1 and Corollary 5.7.2]{bakry2014analysis} since $\BB$ is positive
semidefinite so $H$ is convex and $V_\hh$ is $\alpha$-strongly convex
by construction. 

We initialize the dynamics at infinite temperature
$\mu^{(0)}_\beta=\mu_0$, and denote by $\langle\,\cdot\,\rangle_\beta^{(t)}$ denote
expectation with respect to $\mu_\beta^{(t)}$.
We will use the measures
$\mu_\beta^{(t)}$ as approximations to $\mu_\beta$ 
to prove Theorem~\ref{thm:contloc} via \eqref{eqn:(t)conv}.
The next two propositions formalize steps
{\sf S1} and {\sf S2} in our proof outline above.
\begin{prop}\label{prop:unifconv}
    We have the uniform convergence
    \begin{align*}
\lim_{t\rightarrow\infty}\sup_{(\beta,G) \in [0,1] \times
\cG_{[0,C],C}} \left| \frac1N \langle H(\bphin)\rangle_\beta^{(t)} -
\frac1N \langle H(\bphin)\rangle_\beta\right| = 0.
\end{align*}
\end{prop}

\begin{prop}\label{prop:locenergy}
    For each $\beta\in[0,1]$ and $t\ge0$, $\frac1N\langle
    H(\bphin)\rangle_\beta^{(t)}$ is a local function of
    $G=(\BB,\hh)$. Furthermore, $\beta\mapsto \frac1N\langle
    H(\bphin)\rangle_\beta^{(t)}$ is a norm-continuous function from
    $[0,1]$ to $\cL(\cG_{[0,C],C})$.
\end{prop}

\begin{proof}[Proof of Theorem~\ref{thm:contloc}]
    Recall the interpolation scheme
    \begin{align*}
F_V(G) = F_G(1) = F_G(0) + \int_0^1 F_G'(\beta)\de\beta =
F_V(\bzero,\hh) - \int_0^1 \frac1N\langle H(\bphin)\rangle_\beta
\de\beta.
\end{align*}
    By Proposition~\ref{prop:FV0h},
    $F_V(\bzero,\hh)\in\cL(\cG_{[0,C],C})$. By
    Proposition~\ref{prop:locenergy}, for all $t>0$ the function
    \begin{align*}
\int_0^1\frac1N\langle H(\bphin)\rangle_\beta^{(t)} \de\beta
\end{align*}
is in $\cL(\cG_{[0,C],C})$. By Proposition~\ref{prop:unifconv}, we
have the convergence
    \begin{align*}
\int_0^1 \frac1N \langle H(\bphin)\rangle_\beta^{(t)} \de\beta
\quad\stackrel{t\rightarrow\infty} \longrightarrow\quad \int_0^1
\frac1N \langle H(\bphin)\rangle_\beta \de\beta
\end{align*}
    in $\cL(\cG_{[0,C],C})$. Combining these gives that
    $F_V(G)\in\cL(\cG_{[0,C],C})$, as desired.
\end{proof}

We prove these propositions in the following two subsections. In the intermediate results that
follow, all constants are uniform over $G=(\BB,\hh)\in\cG_{[0,C],C}$
and depend only on $C$ and the $\alpha$-good family $V$ of potentials.
In particular, there is no dependence on the number of vertices $N$ of
$G$.

\subsection{Uniform convergence of the energy: Proof of Proposition~\ref{prop:unifconv}}
\label{sec:UnifConv}

We first show a bound on the second moment of $\mu_\beta$ that, among
other things, will allow us to deduce Lemma~\ref{lem:FVlip}. For
technical reasons later, we will prove this bound for all $\beta$ in
$[0,2]$ rather than $[0,1]$. (Let us emphasize that this lemma does not require
a bound on $\|\BB\|_{\infty\rightarrow\infty}$.)
\begin{lem}\label{lem:secondmoment}
    Assume $\BB$ has spectrum contained in $[0,C]$ and $\|\hh\|_\infty \le C$,
    and that $V$ is an $\alpha$-good  potential.
    There is a positive constant $K=K(C,\alpha,V)$ such that
    $\langle\|\bphin\|^2\rangle_\beta\le KN$ for all $\beta\in[0,2]$.
\end{lem}

\begin{proof}
    Let $\tilde\BB$ be a symmetric matrix with spectrum contained in
    $[\tilde C^{-1},\tilde C]$ for some $\tilde C>0$, and let $\tilde
    H(\bphi)=\<\bphi, \tilde\BB\bphi\>/2$. For $M>0$, let $\cV_M$ denote
    the set of convex functions $V:\R\rightarrow\R$ such that
    \begin{align*}
\frac{\int_\R \varphi^2 \exp(-V(\varphi)) \de\varphi}{\int_\R
\exp(-V(\varphi)) \de\varphi} \le M.
\end{align*}
    Given $V_1,\ldots,V_N\in\cV_M$, let $\tilde
    V(\bphi)=\sum_{i=1}^NV_i(\varphi_i)$. First, note that
    \begin{align*}
        \frac{\int_{\R^N} \|\bphi\|^2 \exp(-\beta\tilde
        H(\bphi)-\tilde V(\bphi))\de\bphi} {\int_{\R^N}
        \exp(-\beta\tilde H(\bphi)-\tilde V(\bphi))\de\bphi} &\le
        \frac{2\tilde C\int_{\R^N} \tilde H(\bphi) \exp(-\beta\tilde
        H(\bphi)-\tilde V(\bphi))\de\bphi} {\int_{\R^N}
        \exp(-\beta\tilde H(\bphi)-\tilde V(\bphi))\de\bphi} \\ &\le
        \frac{2\tilde C\int_{\R^N} \tilde H(\bphi) \exp(-\tilde
        V(\bphi))\de\bphi} {\int_{\R^N} \exp(-\tilde V(\bphi))\de\bphi} \\
        &\le \frac{\tilde C^2 \int_{\R^N} \|\bphi\|^2 \exp(-\tilde
        V(\bphi))\de\bphi} {\int_{\R^N} \exp(-\tilde V(\bphi))\de\bphi} \\
        &= \tilde C^2 \sum_{i=1}^N \frac{\int_\R \varphi^2
        \exp(-V_i(\varphi)) \de\varphi}{\int_\R \exp(-V_i(\varphi))
        \de\varphi} \\ &\le \tilde C^2MN.
    \end{align*}
    for all $\beta\ge0$. Here the second inequality follows from the
    fact that
    \begin{align*}
-\frac{\de}{\de\beta} \left( \frac{\int_{\R^N} \tilde H(\bphi)
\exp(-\beta\tilde H(\bphi)-\tilde V(\bphi))\de\bphi} {\int_{\R^N}
\exp(-\beta\tilde H(\bphi)-\tilde V(\bphi))\de\bphi} \right) =\Var_{\tilde\mu}(\tilde H(\bphin))\ge 0\, ,
\end{align*}
where $\tilde\mu$ is the probability measure on $\R^N$ with density proportional to
$\exp(-\beta\tilde H(\bphi)-\tilde V(\bphi))$.

    Returning to bounding $\langle\|\bphin\|^2\rangle_\beta$, from
    the above it follows that
    \begin{align*}
        \langle\|\bphin\|^2\rangle_\beta &= \frac{\int_{\R^N}
        \|\bphi\|^2 \exp(-\beta H(\bphi) - V_\hh(\bphi))\de\bphi}
        {\int_{\R^N} \exp(-\beta H(\bphi)- V_\hh(\bphi))\de\bphi} \\ &=
        \frac{\int_{\R^N} \|\bphi\|^2 \exp(-\beta (H(\bphi) +
        \frac\alpha8\|\bphi\|^2) -
        (V_\hh(\bphi)-\frac{\alpha\beta}8\|\bphi\| ^2))\de\bphi}
        {\int_{\R^N} \exp(-\beta (H(\bphi) + \frac\alpha8\|\bphi\|^2)
        - (V_\hh(\bphi)-\frac{\alpha\beta}8\|\bphi\| ^2))\de\bphi} \\
        &\le \tilde C^2MN
    \end{align*}
    where $\tilde C = \max(\frac C2+\frac \alpha4,\frac 4\alpha)$ and
    $M$ is such that $V_h(\varphi)-\frac{\alpha\beta}8\varphi^2$ is in
    $\cV_M$ for all $h\in[-C,C]$ and $\beta\in[0,2]$. It is clear that
    such an $M$ exists and only depends on $V$ since
    $V_h(\varphi)-\frac{\alpha\beta}8\varphi^2$ is
    $\frac\alpha2$-strongly convex and $\frac{\int_\R \varphi^2
    \exp(-(V_h(\varphi)-\frac{\alpha\beta} 8\varphi^2))
    \de\varphi}{\int_\R \exp(-(V_h(\varphi)-\frac{\alpha\beta}
    8\varphi^2)) \de\varphi}<\infty$ is continuous in $h$ and $\beta$.
    Taking $K=\tilde C^2M$, this concludes the proof.
\end{proof}

\begin{rem}
We remark that the bound 
\begin{align}
\<\|\bphin-\<\bphin\>_{\beta}\|^2\>_\beta\le \frac{N}{\alpha}
\end{align}
holds by a consequence of the Poincar\'e (or Brascamp-Lieb) 
inequality for $\mu_\beta$.
However, the potentially large contribution to the average squared
magnitude from a nonzero mean of $\mu_\beta$ implies that a more
complicated argument is needed to prove Lemma~\ref{lem:secondmoment}. 
Clearly the same proof works for $\beta$
restricted to any compact interval of nonnegative reals.
\end{rem}

\begin{proof}[Proof of Lemma~\ref{lem:FVlip}]
    Let $\langle\,\cdot\,\rangle =\<\,\cdot\,\>_{\beta=1}$ 
    be expectation with respect to the
    probability measure on $\R^N$ with density proportional to
    $\exp(-\frac12\<\bphi, \BB\bphi\>-V_\hh(\bphi))\de\bphi$. By
    standard statistical physics identities, we may compute
    \begin{align*}
        \nabla_{\BB} F_V(\BB,\hh) = -\frac1{2N}\<\bphin\bphin^{\sT}\>\preceq
        \bzero\, .
    \end{align*}
    Taking magnitudes gives the bound
    \begin{align*}
    \big\|\<\AA,\nabla_{\BB} F_V(\BB,\hh)\>\big| 
    \le \|\AA\| \big|\Tr(\nabla_{\BB} F_V(\BB,\hh))\big|\\
    \le \|\AA\| \frac1{2N} \<\|\bphin\|^2\>\le L\|\AA\|\, .
    \end{align*}
    where $L=K/2$ for $K$ the constant from
    Lemma~\ref{lem:secondmoment}. Since the set of positive
    semidefinite $\BB$ with $\|\BB\|\le C$ is convex, integrating this
    bound along the line segment from $\BB_1$ to $\BB_2$ gives the
    desired result.
\end{proof}

We will now prove Proposition~\ref{prop:unifconv}. We will show 
that the contraction \eqref{eqn:KLcontract} in time of the KL
divergence allows us to write $\frac1N\langle
H(\bphin)\rangle_\beta^{(t)}=\frac1N\langle H(\bphin)\rangle_\kappa$
for some $\kappa$ that converges to $\beta$ as $t\rightarrow\infty$.
It is essential here that the observable $H(\bphin)$ is the
Hamiltonian, which allows us to use the theory of exponential families
and I-projections to control the difference between $\frac1N\langle
H(\bphin)\rangle_\beta^{(t)}$ and $\frac1N\langle
H(\bphin)\rangle_\beta$ by certain Bregman divergences associated
with $F_G$.

We will make use of the following basic property of convex functions
$f$. Its proof is given in
Appendix~\ref{subapp:bregman}.

\begin{lem}\label{lem:bregman}
    Let $f:\R\rightarrow\R$ be a twice-differentiable convex function and let $U$ and $W$
    be positive reals with $U\le W$. Suppose that $a$ and $b$ are
    reals such that $|f'(a)-f'(b)|>U$ and $f''(x)\le W$ for all
     $x$ lying between $a$ and $b$ with $|x-a|\le1$. Then it is the case
    that $f(a)-f(b)-(a-b)f'(b)>\frac{U^2}{2W}$.
\end{lem}

\begin{proof}[Proof of Proposition~%
\ref{prop:unifconv}] 
Recall that $\beta\mapsto F_G(\beta)$ is nonincreasing and convex 
since
    \begin{equation}\label{eqn:F'mean}
        F_G'(\beta) = -\frac1N \langle H(\bphin)\rangle_\beta \le 0\, .
    \end{equation}
    and
    \begin{equation}\label{eqn:F''var}
        F_G''(\beta) = \frac1N \Var_\beta(H(\bphin)) =
        \frac1N(\langle H(\bphin)^2\rangle_\beta - \langle H(\bphin)
        \rangle_\beta^2) \ge 0.
    \end{equation}
For any $\beta_1,\beta_2\in\R$ such
that 
$\int_{\R^N}\exp(-\beta_iH(\bphi)-V_\hh(\bphi) )\de\bphi<\infty$, we
have by \eqref{eqn:F'mean} that
    \begin{align}
        D(\mu_{\beta_1}\|\mu_{\beta_2}) &= \left\langle \log\frac
        {\exp(-\beta_1 H(\bphin) - V_\hh(\bphin) - NF_G(\beta_1))}
        {\exp(-\beta_2 H(\bphin) - V_\hh(\bphin) - NF_G(\beta_2))}
        \right\rangle_{\beta_1}
        \label{eqn:KLidentity} \\
        &= NF_G(\beta_2) - NF_G(\beta_1) + (\beta_2-\beta_1) \langle
        H(\bphin)\rangle_{\beta_1} \nonumber \\ &= N(F_G(\beta_2) -
        F_G(\beta_1) - (\beta_2-\beta_1)F_G'(\beta_1)). \nonumber
    \end{align}
    In particular, for all $\beta\in[0,1]$, we have
    \begin{align}
\frac1N D(\mu_0\|\mu_\beta) = F_G(\beta) - F_G(0) - \beta F_G'(0)\le -
F_G'(0)=\frac1N\langle H(\bphin)\rangle_0\le \frac CN\langle
\|\bphin\|^2\rangle_0\le L(C,\alpha,V)\label{eq:LdefBd}
\end{align}
    for all $\beta\in[0,1]$, where $L=CK$ for $K$ the constant from
    Lemma~\ref{lem:secondmoment}.

    Let $S_{\beta,t}$ be the set of probability measures
\begin{align*}
S_{\beta,t} := \Big\{\mu\in\cuP(\R^N): 
\int_{\R^N}H(\bphi)\mu(\de\bphi) = \langle H(\bphin)\rangle_\beta^{(t)}\Big\}\, .
\end{align*}
This is a linear
    family with statistic $H(\bphin)$. Hence the I-projection
\begin{align}\label{eqn:Iprojection}
\hat{\mu}_{\beta,t} = \arg\min_{\mu\in S_{\beta,t}} D(\mu\|\mu_\beta)\, 
\end{align}
exists and is uniquely given by $\hat{\mu}_{\beta,t} = \mu_{\kappa}$
for some $\kappa=\kappa(t,\beta)\in\R$.
In particular,  $\langle H(\bphin)\rangle_\kappa=\langle
    H(\bphin)\rangle_\beta^{(t)}$ and $D(\mu_\kappa\|\mu_\beta)\le
    D(\mu^{(t)}_\beta\|\mu_\beta)$.
    
We next note that $\kappa\ge0$. Indeed,
    the log-Sobolev inequality, in particular Eq.~\eqref{eqn:KLcontract},
    tells us that
    \begin{align*}
D(\mu_\kappa\|\mu_\beta) \le D(\mu_\beta^{(t)}\|\mu_\beta) \le
e^{-\alpha t} D(\mu_0\|\mu_\beta) \le D(\mu_0\|\mu_\beta)
\end{align*}
    while the identity \eqref{eqn:KLidentity} and the convexity of
    $F_G$ imply that $D(\mu_\kappa\|\mu_\beta)$ is decreasing on
    $\kappa\le\beta$, whenever $\mu_\kappa$ is well-defined.

    Combining these bounds gives
    \begin{align}
        F_G(\beta)-F_G(\kappa)-(\beta-\kappa)
        F_G'(\kappa)&=\frac1ND(\mu_\kappa\| \mu_\beta)\le
        \frac1ND(\mu^{(t)}_\beta\|\mu_\beta) \nonumber \\
        &\le\frac1Ne^{-\alpha
        t}D(\mu_0\|\mu_\beta)\le e^{-\alpha t}L.\label{eqn:bregmanbound}
    \end{align}

    Let $M = e^{-\alpha t}L$ and $P = \alpha^{-1}C^2K$ where $K$ is
    the constant from Lemma~\ref{lem:secondmoment}. Note that
    $\mu_\gamma$ satisfies a Poincar\'e inequality with constant
    $\alpha$ for all $\gamma\ge0$,  by the Bakry-\'Emery
    theorem since $\grad^2(\gamma H+V_\hh)\succeq\alpha\II$.
     Hence, by Eq.~\eqref{eqn:F''var} and Lemma~\ref{lem:secondmoment} we have 
     that
    \begin{align}
        F_G''(\gamma) &= \frac1N \Var_\gamma(H(\bphin)) \le \frac{1}{N\alpha}
        \langle\|\grad H(\bphin)\|^2 \rangle_\gamma\nonumber \\ 
        &=
        \frac{1}{N\alpha} \langle\|\BB\bphin\|^2 \rangle_\gamma \le
        \frac{C^2}{N\alpha}  \langle\|\bphin\|^2\rangle_\gamma\le P
        \label{eqn:F''bound}
    \end{align}
    for all $\gamma\in[0,2]$.

    We claim that $|F_G'(\beta)-F_G'(\kappa)|\le\sqrt{2MP}$ for
    $\beta\in[0,1]$ if $M\le P/2$. Supposing otherwise, we may
    apply Lemma~\ref{lem:bregman} with
    $(f,a,b,U,W)=(F_G,\beta,\kappa,\sqrt{2MP},P)$ to contradict
    \eqref{eqn:bregmanbound}. Indeed, since $\kappa\ge0$ and
    $\beta\in[0,1]$, the bound on the second derivative
    $F_G''(\gamma)$ in the hypothesis of Lemma~\ref{lem:bregman} is
    only needed for $\gamma\in[0,2]$ and is provided by
    Eq.~\eqref{eqn:F''bound}.

    On the other hand, we have that $F_G'(\beta)=-\frac1N\langle
    H(\bphin)\rangle_\beta$ and $F_G'(\kappa)=-\frac1N\langle
    H(\bphin)\rangle_\kappa = -\frac1N\langle
H(\bphin)\rangle_\beta^{(t)}$ by Eq.~\eqref{eqn:F'mean} 
    since $\mu_{\kappa}=\hat{\mu}_{\beta,t}$ 
    is the I-projection of  Eq.~\eqref{eqn:Iprojection}.
     It follows that
    \begin{align*}
\left| \frac1N\langle H(\bphin)\rangle_\beta^{(t)} - \frac1N\langle
H(\bphin)\rangle_\beta \right| \le \sqrt{2MP}= \sqrt{2C^2KL\alpha^{-1}}
\cdot e^{-\alpha t/2}\, ,
\end{align*}
    for $\beta\in[0,1]$ if $e^{-\alpha t}\le C^2 K/(2\alpha L)$. 
    By Lemma \ref{lem:secondmoment} and Eq.~\eqref{eq:LdefBd}, 
    the factor $\sqrt{2C^2KL\alpha^{-1}}$ only depends on $C,\alpha,V$
    and not on the graph $G$, whence 
    \begin{align*}
\lim_{t\rightarrow\infty} \sup_{\beta\in[0,1], G\in\cG_{[0,C],C}}
\left| \frac1N\langle H(\bphin)\rangle_\beta^{(t)} - \frac1N\langle
H(\bphin)\rangle_\beta \right| = 0,
\end{align*}
    as desired.
\end{proof}

\begin{rem}
    As mentioned already, the results of this subsection do not assume the sparsity
    condition $\|\BB\|_{\infty\rightarrow\infty}\le C$ and 
    apply to all graphs $(\BB,\hh)$ with $\bzero\preceq \BB\preceq
    C\II$ and $\|\hh\|_\infty\le C$. However, the remainder of the
    proof uses sparsity in a crucial way.
\end{rem}

\subsection{Polynomial approximation of
discretized dynamics: Proof of Proposition~\ref{prop:locenergy}}\label{subs:discretize}

In order to prove Proposition~\ref{prop:locenergy}, we will write $
\langle H(\bphin)\rangle_\beta^{(t)}/N$ as the (uniform over
$[0,1]\times\cG_{[0,C],C}$) limit as $k\rightarrow\infty$ of
$Q_k(\beta,\BB,\hh)$ for some polynomial function $Q_k$ of $\beta$ and
the entries of $\BB$ and $\hh$. These functions $Q_k$ will have the
property that they can be rewritten as linear combinations of weighted
homomorphism densities $t(X,G)$ with coefficients that are polynomials
in $\beta$. As a consequence, we will show
that $\beta\mapsto Q_k(\beta,\BB,\hh)$ is
bounded and continuous (with respect to the uniform norm in 
$\cL(\cG_{[0,C],C})$),
and hence $\beta
\mapsto  \langle H(\bphin) \rangle_\beta^{(t)}/N$ is
bounded and continuous.

To construct the approximants $Q_k(\beta,\BB,\hh)$, 
we will discretize the Langevin SDE
\eqref{eqn:langevin}, which yields a discrete time recursion that we
further approximate using polynomials of $\beta,\BB,\hh$. 
We will use
the following fact about the classical Euler-Maruyama scheme. Its
proof is straightforward but we will need an estimate
that tracks the dependence on the dimension $N$, which we could not
locate elsewhere, so we give it in Appendix~\ref{subapp:discretize}.

\begin{prop}\label{prop:discretize}
    Given a strong solution $t\mapsto \xx^{(t)}\in\R^N$ to the SDE
    \begin{align*}
\de\xx^{(t)}=\aa(\xx^{(t)})\de t+\sigma\de\ww^{(t)}
\end{align*}
    on $\R^N$ for some function $\aa:\R^N\rightarrow\R^N$ and constant diffusion coefficient $\sigma$, along with a step
    size $\delta\in(0,1)$, consider the Euler-Maruyama scheme
    $\yy^{(0)},\yy^{(1)},\yy^{(2)},\ldots$ defined recursively by
    $\yy^{(0)}=\xx^{(0)}$ and
    \begin{align*}
\yy^{(n+1)}=\yy^{(n)}+\delta
\aa(\yy^{(n)})+\sigma(\ww^{(\delta(n+1))}-\ww^{(\delta n)}).
\end{align*}
    Suppose that, for some positive constants $K_1,K_2,K_3$, we have that
    $\E[\|\xx^{(0)}\|^2]\le K_1N$, $\| \aa(\bzero)\|^2\le K_2N$, and
    $\| \aa(\uu)-\aa(\vv)\|\le K_3\|\uu-\vv\|$ for all $\uu,\vv\in\R^N$,
    i.e. the initial condition is bounded in $L^2$ and the drift
    coefficient is Lipschitz and bounded at the origin. Then there
    exists a constant $L$ depending only on $K_1,K_2,K_3$ such that
    for all $\delta>0$ and $T\in\delta\N$, we have that
    \begin{align*}
\E[\|\xx^{(T)}-\yy^{(T/\delta)}\|^2]\le Le^{LT}\delta N.
\end{align*}
\end{prop}

\begin{rem}
    One might wonder whether it is necessary to reduce to a continuous
    Gibbs measure as we have done in Section~\ref{subs:renormalize} in the
    first place. Indeed, one might try to work directly with Glauber
    dynamics for the Ising model using the log-Sobolev inequality
    established in \cite{AJKPV22} under \ref{cond:SW} instead of
    renormalizing and using Langevin dynamics. However,
    the direct analogue of Proposition~\ref{prop:discretize}
    fails in this approach.
\end{rem}

Returning to $ \langle H(\bphin)\rangle_\beta^{(t)}/N$, consider
the Euler-Maruyama iteration
$\btau^{(0)},\btau^{(1)},\ldots,\btau^{(t/\delta)} $ for the Langevin
SDE \eqref{eqn:langevin} given by Proposition~\ref{prop:discretize}.
Implicit from here on is that whenever a choice for $\delta$ is made,
$\delta$ is chosen so that $t/\delta$ is a positive integer. 
We will not comment on this point later.

Note that the drift $-(\beta\grad H+\grad V_\hh)$ and law
$\mu_0$ of the initial condition $\btau^{(0)}=\bphin^{(0)}$ 
in Eq.~\eqref{eqn:langevin}  satisfy
the hypotheses of the previous proposition with constants
$K_1,K_2,K_3$ uniform over $(\beta,G)\in[0,1]\times\cG_{[0,C],C}$
depending only on $C$ and $\alpha$ (as long as $V$ is $\alpha$-good). Applying
Proposition~\ref{prop:discretize}, we conclude that for all
$\eps>0$ there exists a choice of $\delta =\delta(t,\eps,C,V)$ so that 
\begin{align*}  
\E[
\|\bphin^{(t)} - \btau^{(t/\delta)}\|^2 ] \le N\eps\;\;\;\;\;\;
\forall (\beta,G)\in[0,1]\times\cG_{[0,C],C}\, .
\end{align*}

We then have that
\begin{align*}
    |\langle H(\bphin)\rangle_\beta^{(t)} -
    \E[H(\btau^{(t/\delta)})]| &= |\E[
    \<\bphin^{(t)}-\btau^{(t/\delta)}, \BB\bphin^{(t)}\> ] - \E[
    H(\bphin^{(t)}-\btau^{(t/\delta)}) ]| \\ &\le C\E[
    \|\bphin^{(t)} - \btau^{(t/\delta)}\|^2 ]^{1/2}
    (\langle\|\bphin\|^2\rangle^{(t)}_\beta) ^{1/2} + \frac C2 \E[
    \|\bphin^{(t)} - \btau^{(t/\delta)}\|^2 ] \\ &\le
    \left(C\eps^{1/2}K^{1/2}+\frac12 C\eps\right)N
\end{align*}
for all $(\beta,G)\in[0,1]\times\cG_{[0,C],C}$ where $K$ is given by
Lemma~\ref{lem:secondmoment}. We may take $\eps$, and hence
$C\eps^{1/2}K^{1/2}+\frac12 C\eps$, to be arbitrarily small,
uniformly over $[0,1]\times\cG_{[0,C],C}$, by choosing $\delta$
appropriately. Thus, it suffices to prove that $\beta\mapsto
\E[H(\btau^{(t/\delta)})]/N$ is a norm-continuous map from $[0,1]$ to
$\cL(\cG_{[0,C],C})$ for every $t,\delta>0$ so that $t/\delta\in\N$.

In order to prove that $\beta\mapsto\frac1N
\E[H(\btau^{(t/\delta)})]$  is continuous as a map 
$\R\to\cL(\cG_{[0,C],C})$, we write the discretized dynamics as
\begin{equation}\label{eqn:taurecursion}
    \btau^{(n+1)} = \btau^{(n)} - \beta\delta \BB \btau^{(n)}
    - \delta \grad V_\hh(\btau^{(n)})+\sqrt{2\delta}\bet^{(n)}
\end{equation}
where $\bet^{(0)},\bet^{(1)},\ldots$ are independent standard
Gaussians in $\R^N$. Note that the $i$th component of $\grad
V_\hh(\btau^{(n)})$ is equal to $V_{h_i}'(\tau^{(n)}_i)$. 
Since $V$ is $\alpha$-good, the map $(h,\varphi)\mapsto
V_h'(\varphi)$ is continuous on $[-C,C]\times\R$ and thus can be
locally uniformly approximated by polynomials. Our strategy will be to
choose suitable polynomial approximations $P^{(n)}(h,\varphi)$ of this
map and consider the dynamics given recursively by
$\bpsi^{(0)}=\btau^{(0)}$ and
\begin{equation}\label{eqn:psirecursion}
    \bpsi^{(n+1)} = \bpsi^{(n)} - \beta\delta \BB \bpsi^{(n)}
    - \delta \PP^{(n)}(\hh,\bpsi^{(n)}) + \sqrt{2\delta}\bet^{(n)}
\end{equation}
where the $i$th component of $\PP^{(n)}(\hh,\bpsi^{(n)})$ is
$P^{(n)}(h_i,\psi^{(n)}_i)$, i.e.
$\PP^{(n)}:\R^N\times\R^N\rightarrow\R^N$ is $P^{(n)}$ applied to each
component. We will show that it is possible to choose $P^{(n)}$ such
that $\E[H(\btau^{(t/\delta)})]$ is well approximated by 
$\E[H(\bpsi^{(t/\delta)})]$. Since $H(\bpsi^{(t/\delta)})$ is a polynomial in
$\bpsi^{(0)}$ and the noise $\bet^{(0)},\bet^{(1)},\ldots$, this will
yield a polynomial approximation to
$\E[H(\btau^{(t/\delta)})]/N$.

We will begin by bounding the sub-Gaussian norms of the coordinates
of $\btau^{(n)}$. Following, e.g. \cite[Exercise 2.40]{Ver26}, for a
sub-Gaussian random variable $X$, we define the sub-Gaussian variance
\begin{align*}
\Var_G(X) = \inf \left\{ \sigma^2 : \E \left[e^{\lambda(X-\E[X])}
\right] \le e^{\sigma^2 \lambda^2 / 2} \text{ for all } \lambda\in\R
\right\}
\end{align*}
and the exact sub-Gaussian norm
\begin{align*}
\|X\|_G = \sqrt{\Var_G(X) + \E[X]^2}.
\end{align*}
Furthermore, $\|\cdot\|_G$ satisfies the triangle
inequality, i.e. $\|X+Y\|_G\le\|X\|_G+\|Y\|_G$ for all $X$ and $Y$,
as well as the improved estimate
$\|X+Y\|_G^2\le\|X\|_G^2+\|Y\|_G^2$ if $X$ and $Y$ are independent and
$\E[Y]=0$. It is also the case that $\|X\|_G\ge\E[X^2]^{1/2}$, from
which we may conclude that if $f:\R\rightarrow\R$ is an $L$-Lipschitz
function with $|f(0)|\le L$, then
\begin{align*}
\|f(X)\|_G\le\sqrt{\Var_G(f(X))} + \E[|f(X)|] \le 18L\|X\|_G + (L +
L\E[|X|]) \le 20L(1+\|X\|_G).
\end{align*}
%

For a sub-Gaussian random vector $\xx\in\R^N$, let
\begin{align*}
\|\xx\|_{G,\infty} = \sup_{1\le i\le N}\|x_i\|_G
\end{align*}
be the maximum exact sub-Gaussian norm of a coordinate of $\xx$. We
have the following basic properties of $\|\cdot\|_{G,\infty}$ by the
previous discussion.

\begin{enumerate}
    \item By the triangle inequality for the exact
    sub-Gaussian norm, we have that
    \begin{align*}
\|\xx+\yy\|_{G,\infty} \le \|\xx\|_{G,\infty} + \|\yy\|_{G,\infty}.
\end{align*}
    \item For any deterministic matrix $\AA\in\R^{N\times N}$, 
    we have that
    \begin{align*}
\|\AA\xx\|_{G,\infty} \le \|\AA\|_{\infty\rightarrow\infty}
\|\xx\|_{G,\infty}\, .
\end{align*}
    \item Given $L$-Lipschitz functions
    $f_1,\ldots,f_N:\R\rightarrow\R$ with $|f_1(0)|,\ldots|f_N(0)|\le
    L$ and their coordinate-wise application
    $\ff:\R^N\rightarrow\R^N$, we have that
    \begin{align*}
\|\ff(\xx)\|_{G,\infty} \le 10L(1 + \|\xx\|_{G,\infty}).
\end{align*}
\end{enumerate}

Similar to before, in what follows all constants in the statements of
intermediate steps are uniform over $[0,1]\times\cG_{[0,C],C}$ and
depend only on $C,V,\delta,t$.
\begin{lem}\label{lem:psi2inftybound}
    There exists a constant $M$, depending only on $C,V,\delta,t$, such that
    $\|\btau^{(k)}\|_{G,\infty}\le M$ for all $k\le t/\delta$.
\end{lem}

\begin{proof}
    We proceed by induction on $k$ to show that
    $\|\btau^{(k)}\|_{G,\infty}\le M_k$ for some constant $M_k$
    depending only on $C,V,\delta$. For the base case of $k=0$, note
    that $\btau^{(0)}$ has law $\mu_0$. In particular, the coordinates
    of $\btau^{(0)}$ are independent with densities proportional to
    $e^{-V_h(\varphi)}\de\varphi$ for some $h\in[-C,C]$. Since
    $V_h''(\varphi)\ge\alpha$ for all $h$, $\tau^{(0)}_i$ is
    $\alpha$-strongly log-concave and we have that
    $\|\tau^{(0)}_i\|_G\le|\E[\tau_i^{(0)}]| +10\alpha^{-1/2}$. The
    continuity of $V_h$ in $h$ and compactness of $[-C,C]$ give a
    uniform bound on $|\E[\tau_i^{(0)}]|$ depending only on $C$ and
    $V$ so $\|\btau^{(0)}\|_{G,\infty}\le M_0$ for some constant
    $M_0$, completing the base case.

    For the induction step, by definition of $\alpha$-good
    potential the map $\varphi\mapsto V_h'(\varphi)$ is $\alpha^{-1}$-Lipschitz,
    and since $h\mapsto V_h'$ is continuous, 
    $\sup_{h\in[-C,C]}|V_h'(0)|\le L$, for some constant $L$. We will take without loss of generality $L\ge \alpha^{-1}$.
 Since the noise $\bet^{(n)}$ has mean $0$ and is
    independent of $\btau^{(n)}$ for all $n$, we have by
    \eqref{eqn:taurecursion} that
    \begin{align*}
\|\btau^{(k+1)}\|_{G,\infty}^2 \le \left\| \btau^{(k)} -
\beta\delta \BB \btau^{(k)} - \delta \grad
V_\hh(\btau^{(k)}) \right\|_{G,\infty}^2 + 2\delta
\end{align*}
    as the exact sub-Gaussian norm of a standard Gaussian is $1$. The
    basic properties of $\|\cdot\|_{G,\infty}$ then allow us to bound
    \begin{align*}
\left\| \btau^{(k)} - \beta\delta \BB \btau^{(k)} -
\delta\grad V_\hh(\btau^{(k)}) \right\|_{G,\infty} \le
\|\btau^{(k)}\|_{G,\infty} + \delta
\|\BB\|_{\infty\rightarrow\infty} \|\btau^{(k)}\|_{G,\infty} +
10\delta L(1 + \|\btau^{(k)}\|_{G,\infty}),
\end{align*}
    which means that $\|\tau^{(k+1)}\|_{G,\infty}^2 \le M_{k+1}$ for
    the constant
    \begin{equation}\label{eqn:Mkrecurse}
        M_{k+1} = \sqrt{\left(M_k + C\delta M_k + 20\delta L(1
        + M_k)\right)^2 + 2\delta}
    \end{equation}
    by the induction hypothesis. Here, we are using $\beta\le1$ and
    the sparsity bound $\|\BB\|_{\infty\rightarrow\infty}\le C$.
    Noting $M_{k+1}\ge M_k$ and taking $M=M_{t/\delta}$, the lemma is
    proved.
\end{proof}

\begin{rem}
    By a more careful analysis of \eqref{eqn:Mkrecurse}, we can show
    that the constant $M$ in Lemma~\ref{lem:psi2inftybound} can be
    taken to be uniform in $\delta\in(0,\delta_0)$ for some
    $\delta_0>0$ and depend only on $C,V,t$ for sufficiently small
    $\delta$. By increasing $M_0$ if necessary, we may assume that
    $M_k\ge1$ for all $k$. Then, we may bound
    \begin{align*}
M_k + C\delta M_k + 20\delta L(1 + M_k) \le
M_k(1+C\delta+40L\delta)\le M_k(1+\tilde C\delta)
\end{align*}
    for $\tilde C = C+40L$. As $(1+\tilde C\delta)^2\le 1+3\tilde
    C\delta$ for sufficiently small $\delta$, we then have the bound
    $M_{k+1}^2 \le M_k^2(1+3\tilde C\delta) + 2\delta$. Iterating gives
    \begin{align*}
M_{t/\delta}^2 \le M_0^2(1+3\tilde C\delta)^{t/\delta} +
\delta\sum_{k=0}^{t/\delta-1}(1+3\tilde C\delta)^k\le M_0^2e^{3\tilde
Ct}+\frac{e^{3\tilde Ct}-1}{3\tilde C},
\end{align*}
which is indeed independent of sufficiently small $\delta$.
\end{rem}

Recall that our plan is to approximate the gradient term $\grad V_\hh$
in \eqref{eqn:taurecursion} by a polynomial. To construct this
polynomial, we will need the following consequence of weighted
polynomial approximation of continuous functions on $\R\times[-C,C]$. It
seems to be a direct generalization of standard results
\cite{lubinsky2007survey} but we could not locate its exact form elsewhere,
so we prove it in Appendix~\ref{subapp:wtapprox},

\begin{prop}\label{prop:wtapprox}
    For all $A,\eps>0$, there exists a polynomial
    $P\in\R[h,\varphi]$ such that
    \begin{align*}
\sup_{(h,\varphi)\in[-C,C]\times\R} e^{-A\varphi^2}
|V_h'(\varphi)-P(h,\varphi)| < \eps\, .
\end{align*}
\end{prop}

\begin{lem}\label{lem:subgapprox}
    For all $M,\eps>0$ and $p\ge1$, there exists a polynomial
    $P\in\R[h,\varphi]$ such that $\E[|V_h'(X)-P(h,X)|^p]<\eps$
    for all $h\in[-C,C]$ and random variables $X$ with $\|X\|_G\le M$.
\end{lem}

\begin{proof}
    We have the bound
    \begin{align*}
\E[|V_h'(X)-P(h,X)|^p] \le \E[e^{ApX^2}] \cdot
\left(\sup_{(h,\varphi)\in[-C,C]\in\R} e^{-A\varphi^2} |V_h'(\varphi)
- P(h,\varphi)| \right)^p.
\end{align*}
    Choosing $A<\frac1{9pM^2}$ so that $\E[e^{ApX^2}]\le2$, the result
    follows from Proposition~\ref{prop:wtapprox} and the fact that
    $\|X\|_{\psi_2}\le 3\|X\|_G\le 3M$.
\end{proof}

Applying Lemma~\ref{lem:subgapprox}, we will henceforth fix a sequence
$P_1,P_2,\ldots\in\R[h,\varphi]$ of polynomials such that
$\E[|V_h'(X)-P_{\ell}(h,X)|^{\ell}]^{1/\ell}<1/\ell$ for all $h\in[-C,C]$ and
random variables $X$ with $\|X\|_G\le \ell$. We will let
$\PP_{\ell}:\R^N\times\R^N\rightarrow\R^N$ denote the component-wise
application of $P_\ell$. Given a polynomial $P:\R^N\rightarrow\R$, let
$\|P\|_1$ denote the sum of the magnitudes of the coefficients of $P$.

We are now in a position to bound the difference between
$\E[H(\btau^{(t/\delta)})]$ and $\E[H(\bpsi^{(t/\delta)})]$ for
$\btau^{(t/\delta)}$ and $\bpsi^{(t/\delta)}$ given by
\eqref{eqn:taurecursion} and \eqref{eqn:psirecursion}. The overall
idea is to integrate out the randomness in $\bet^{(n)}$ backwards from
$n=t/\delta$ to $n=0$ and use induction. In order to make the
induction work, we will generalize to bounding the difference between
polynomial observables $\E[P(\btau^{(t/\delta)})]$ and
$\E[P(\bpsi^{(t/\delta)})]$ for the class of sparse polynomials $P$
with $\|P\|_1=O(N)$.

In what follows, $\BB$ and $\hh$ will be regarded as fixed so their
entries will contribute only to coefficients of polynomials of
$\btau^{(k)}$ and all choices will be made uniformly over
$\cG_{[0,C],C}$. This argument relies on the sparsity condition
\ref{cond:SG} to ensure both that integrating backwards in time
preserves this class and that our observable of interest $H$ lies in
this class.
%
\begin{lem}\label{lem:backprop}
    Let $d$ be a positive integer and $D,\eps$ be positive reals.
    There exists a positive integer $\ell$ depending only on
    $d,D,\eps,C,V,\delta,t$ such that for all $k<t/\delta$ and
    polynomials $P:\R^N\rightarrow\R$ of degree at most $d$ with
    $\|P\|_1\le DN$, we have that
    \begin{align*}
 \left|\E[ P(\btau^{(k+1)}) ] - \E\left[ P\left( \btau^{(k)} -
 \beta\delta \BB\btau^{(k)} -\delta
 \PP_\ell(\hh,\btau^{(k)}) + \sqrt{2\delta}\bet^{(k)} \right) \right]
 \right| \le \eps N.
\end{align*}
    for all $\hh$ with $\|\hh\|_\infty\le C$. Furthermore, there exist
    a positive integer $d'$ and a positive real $D'$ depending only
    $d,D,\eps,C,V,\delta,t$ such that for all polynomials
     $P:\R^N\rightarrow\R$ of degree at most $d$ with
    $\|P\|_1\le DN$, we have that
    \begin{align*}
 \E\left[ P\left( \btau^{(k)} - \beta\delta\BB\btau^{(k)} -
 \delta\PP_\ell(\hh,\btau^{(k)}) + \sqrt{2\delta}\bet^{(k)}
 \right) \Big| \btau^{(k)} \right] = Q(\btau^{(k)})
\end{align*}
    for some polynomial $Q:\R^N\rightarrow\R$ of degree at most $d'$
    with $\|Q\|_1\le D'N$.
\end{lem}

\begin{proof}
    It suffices to show that for all $\eps_0>0$, we may choose
    $\ell$ so that
    \begin{align*}
 \left|\E[ R(\btau^{(k+1)}) ] - \E\left[ R\left( \btau^{(k)} -
 \beta\delta \BB\btau^{(k)} - \delta
 \PP_\ell(\hh,\btau^{(k)}) + \sqrt{2\delta}\bet^{(k)} \right)
 \right]\right| \le \eps_0
\end{align*}
    whenever $R:\R^N\rightarrow\R$ is a monic monomial of degree at
    most $d$ and $\|\hh\|_\infty\le C$. Indeed, summing over all terms
    of $P$ and using $\|P\|_1\le DN$ gives the desired conclusion when
    $\eps_0<\frac\eps D$.

    To choose such a suitable $\ell$, let $X_1,\ldots,X_m$ be the (not
    necessarily distinct) $\deg R=m\le d$ coordinates of
    $\btau^{(k+1)}$ corresponding to the factors of $R$, and let
    $Y_1,\ldots,Y_m$ be the corresponding coordinates of
    \begin{align*}
\btau^{(k)}-\beta\delta\BB\btau^{(k)}
-\delta\PP_\ell(\hh,\btau^{(k)}) +\sqrt{2\delta}\bet^{(k)}.
\end{align*}
    We have by H\"older's inequality that
    \begin{align*}
\E[ |X_1\cdots X_m - Y_1\cdots Y_m| ] \le
\sum_{S\subsetneq\{1,\ldots,m\}} \prod_{i\in S} \E[|X_i|^m]^{1/m}
\prod_{j\notin S} \E[ |X_j-Y_j|^m ]^{1/m}.
\end{align*}
    By Lemma~\ref{lem:psi2inftybound}, we have for all $i\le m$ and
    $k\le t/\delta$ that $\|X_i\|_G\le \|\btau^{(k+1)}\|_{G,\infty}\le
    M$ for some constant $M$ depending only on $C,V,\delta,t$. Since
    \begin{align*}
\E[|X_i|^m]^{1/m} \le 2m^{1/2} \|X_i\|_{\psi_2} \le 6m^{1/2}\|X_i\|_G
\le 6dM
\end{align*}
    and
    \begin{align*}
\E[ |X_j-Y_j|^m ]^{1/m} = \delta \E[ |V_{h_j}'(\tau^{(k)}_j) -
P_\ell(h_j,\tau^{(k)}_j)|^m ]^{1/m} \le \frac\delta{\ell}
\end{align*}
    for all $\ell>\max(m,M)$ by the construction of $P_\ell$, we may
    choose $\ell$, depending only on $d,\eps_0,M$, so that
    \begin{align*}
|\E[X_1\cdots X_m]-\E[Y_1\cdots Y_m]|\le \E[|X_1\cdots X_m-Y_1\cdots
Y_m|]\le 2^d(1+6dM)^d \frac\delta{\ell} \le \eps_0,
\end{align*}
as desired.

    For the second part, it is clear that $Q$ is a polynomial in $\hh$
    and $\btau^{(k)}$ whose degree is at most $\deg P\cdot(1+\deg
    P_\ell)$, so it remains to show that $\|Q\|_1\le D'N$ for some
    $D'$. By the same argument as before, it suffices to show the
    analogous result for monic monomials $R$, i.e. that for all such
    $R:\R^N\rightarrow\R$ of degree at most $d$, we have that
    \begin{align*}
\E\left[ R\left( \btau^{(k)} - \beta\delta \BB\btau^{(k)} -
\delta \PP_\ell(\hh,\btau^{(k)}) + \sqrt{2\delta} \bet^{(k)}
\right) \Big| \btau^{(k)} \right] = S(\btau^{(k)})
\end{align*}
    for $S:\R^N\rightarrow\R$ a polynomial with $\|S\|_1\le D_0$,
    where $D_0$ depends only on $d,D,\eps,C,V,\delta,t$.

    Let us briefly consider the polynomial
    \begin{align*}
R_0(\btau^{(k)},\bet^{(k)}) = R \left( \btau^{(k)} -
\beta\delta \BB\btau^{(k)} - \delta
\PP_\ell(\hh,\btau^{(k)}) + \sqrt{2\delta}\bet^{(k)} \right)
\end{align*}
    of the coordinates of $\btau^{(k)}$ and $\bet^{(k)}$. Suppose that
    $R(\xx) = \prod_{i=1}^mx_{a_i}$ for
    $a_1,\ldots,a_m\in\{1,\ldots,N\}$ with $m\le d$. We then have that
    \begin{align*}
        \|R_0\|_1 &\le \prod_{i=1}^m \left\| \left( \btau^{(k)} -
        \beta\delta \BB\btau^{(k)} - \delta
        \PP_\ell(\hh,\btau^{(k)}) + \sqrt{2\delta}\bet^{(k)}
        \right)_{a_i}\right\|_1 \\ &\le \prod_{i=1}^m \left( 1 +
        \beta\delta \|\BB\|_{\infty\rightarrow\infty} +
        \delta\|P_\ell(h_{a_i},\tau^{(k)}_{a_i})\|_1 +
        \sqrt{2\delta} \right),
    \end{align*}
     which is clearly a quantity bounded in terms of
     $d,D,\eps,C,V,\delta,t$ since $P_1,P_2,\ldots$ is a fixed
     sequence of two-variable polynomials depending only on $C$ and
     $V$, and $\ell$ is a function of $d,D,\eps,C,V,\delta,t$. On
     the other hand, evaluating the conditional expectation $\E[
     R_0(\btau^{(k)},\bet^{(k)}) \mid \btau^{(k)} ] = S(\btau^{(k)})$
     amounts to substituting in Gaussian moments for the terms
     containing components of $\bet^{(k)}$. The $m$th moments of
     $\bet^{(k)}$ for $m\le d$ are crudely bounded above in magnitude
     by $d!$, so $\|S\|_1\le d!\|R_0\|_1$ is bounded in terms of
     $d,D,\eps,C,V,\delta,t$. This concludes the proof.
\end{proof}

With Lemma~\ref{lem:backprop}, we can now
prove that the polynomial dynamics \eqref{eqn:psirecursion} 
is a good approximation of the discrete Langevin dynamics \eqref{eqn:taurecursion},
in the following sense. 
\begin{lem}\label{lem:afterbackprop}
For any $\eps>0$, there exist polynomials $(\P^{(n)}_\eps)_{n\ge 0}$,
depending only on $t,\delta,V,C$,  for the discrete polynomial dynamics
\eqref{eqn:psirecursion}, such that, for all $N\ge 1$, $\beta\in[0,1]$ and $G_N\in\cG_{[0,C],C}$
\begin{align}
\frac{1}{N}
\left|\E[H(\btau^{(t/\delta)})] -
\E[H(\bpsi^{(t/\delta)})]\right| \le \eps\, .
\end{align}
\end{lem}
\begin{proof}
Letting $K=t/\delta$, we will take $\PP^{(k)}_\eps=\PP_{\ell_k}$ for $\ell_1,\dots \ell_K$
chosen so that the error incurred by replacing $\grad
V_\hh$ by $\PP_{\ell_k}$ for all $k$ is small. 
Define
\begin{equation}\label{eqn:Hk}
    H_k(\bpsi^{(k)})=\E[H(\bpsi^{(K)}) \mid\bpsi^{(k)}]\, ,
\end{equation}
and note that, as a consequence,
\begin{align*}
    H_k(\bpsi^{(k)}) = \E[H(\bpsi^{(K)}) \mid \bpsi^{(k)}] = \E[
    \E[H(\bpsi^{(K)}) \mid \bpsi^{(k+1)}] \mid \bpsi^{(k)}] = \E[
    H_{k+1}(\bpsi^{(k+1)}) \mid \bpsi^{(k)}]\, .
   \end{align*}

Along with $\ell_K,\ell_{K-1},\ldots,\ell_0$, we will construct
sequences $d_K,d_{K-1},\ldots,d_0$ and $D_K,D_{K-1},\ldots,D_0$ that
are uniform over $[0,1]\times\cG_{[0,C],C}$ so that 
\begin{align}
0\le j\le K\;\;\Rightarrow\;\; \deg H_j\le d_j\, ,
\|H_j\|_1\le D_jN\, .\label{eq:InductivePolynomialBounds}
\end{align}
We do so by backwards
induction. Since $H_K=H$, the above inequalities hold
with  $d_K=2$ and $D_K=C$ since $H_K=H$. 

Next assume $(\ell_j)_{j\ge k+1}$, $(d_j)_{j\ge k+1}$, and $(D_j)_{j\ge k+1}$
have been chosen so that the inequalities
\eqref{eq:InductivePolynomialBounds} holds for $j\ge k+1$.
We apply Lemma~\ref{lem:backprop} with $(d,D)=(d_{k+1},D_{k+1})$ to
choose $\ell_k$ so that
\begin{equation}\label{eqn:Plerror}
    \left| \E[ H_{k+1}(\btau^{(k+1)}) ] - \E\left[ H_{k+1}\left(
    \btau^{(k)} - \beta\delta \BB\btau^{(k)} - \delta
    \PP_{\ell_k} (\hh,\btau^{(k)}) + \sqrt{2\delta}\bet^{(k)} \right)
    \right] \right| \le \frac{\eps N}K\, .
\end{equation}
In particular, this determines $H_k:\R^N\rightarrow\R$ as the
polynomial given by
\begin{align*}
H_k(\xx) = \E\left[ H_{k+1}\left( \xx - \beta\delta \BB\xx -
\delta \PP_{\ell_k} (\hh, \xx) + \sqrt{2\delta}\bet^{(k)} \right)
\right].
\end{align*}
Lemma~\ref{lem:backprop} then implies that $\deg H_k\le d_k$ and
$\|H_k\|_1\le D_kN$ for $d_k$ and $D_k$ bounded in terms of
$d_{k+1},D_{k+1},\eps,C,V,\delta,t$, which completes the inductive
construction of the $\ell_k$.

With this choice of $\ell_k$, we have from Eq.~\eqref{eqn:Plerror} that
\begin{align*}
\left| \E[H_{k+1}(\btau^{(k+1)})] - \E[H_k(\btau^{(k)})] \right| \le
\frac{\eps N}K
\end{align*}
for all $0\le k<K$, which means by Eqs.~\eqref{eqn:psirecursion} and
\eqref{eqn:Hk} that
\begin{align*}
    \left| \frac1N \E[H(\btau^{(K)})] - \frac1N \E[H(\bpsi^{(K)})]
    \right| &= \left| \frac1N \E[H_K(\btau^{(K)})] - \frac1N
    \E[H_0(\bpsi^{(0)})] \right| \\ &= \left| \frac1N
    \E[H_K(\btau^{(K)})] - \frac1N \E[H_0(\btau^{(0)})] \right| \le \eps.
\end{align*}
\end{proof}

Summarizing, we have shown that the discrete polynomial dynamics
\eqref{eqn:psirecursion} approximates the discrete Langevin dynamics
\eqref{eqn:taurecursion} uniformly over $[0,1]\times\cG_{[0,C],C}$. 
It remains to analyze  the discrete polynomial dynamics \eqref{eqn:psirecursion} 
and prove
that $\E[H(\bpsi^{(K)})]/N\in\cL(\cG_{[0,C],C})$, for $K=t/\delta$ provided
$\PP^{(k)}=\PP_{\ell_k}$ for some choice of $\ell_k$, $k\in\{0,1,\ldots,
K-1\}$. We describe the intuition before passing to the actual proof.

 The main observation is 
that $H(\bpsi^{(K)})$ is a polynomial in
$\BB,\hh,\bpsi^{(0)}$, $(\bet^{(k)})_{k\le K-1}$.
The expectation is taken with respect to the noise variables 
$(\bet^{(k)})_{k\le K-1}$ and the initial condition $\bpsi^{(0)}$
which is distributed according to the product measure $\mu_0$.
Taking the expectation amounts to
substituting in Gaussian moments for each power of $\eta_i^{(k)}$, $i\le N$,
$k\le K-1$ and single-spin
moments for each power of $\psi^{(0)}_i$, $i\le N$.
These single-spin
moments can be further approximated by polynomials in $h_i$ over
$[-C,C]$, after which we arrive at a polynomial expression in
$G=(\BB,\hh)$.

Before providing a rigorous proof, we introduce 
one more piece of terminology. 
\begin{defn}\label{defn:rootedtest}
    A \emph{rooted test graph} is a connected test graph with a
    distinguished root vertex. If $G=(\AA,\hh_1,\ldots,\hh_k)$ is a
    weighted graph on $[N]$ and $X=(\WW,\dd_1,\ldots,\dd_k)$ is a
    rooted test graph on $[n]$ with root vertex $1$, then we define
    the \emph{weighted homomorphism density vector} $\btt(X,G)\in\R^N$
    as the vector whose $m$th component is given by
    \begin{align*}
    t_m(X,G) :=\sum_{\substack{f:[n]\rightarrow[N]\\f(1)=m}} \left( \prod_{i\in[n]}
\prod_{\ell=1}^k h_{\ell,f(i)}^{d_{\ell,i}} \right) \left( \prod_{i\le
j\in[n]} a_{f(i)f(j)}^{w_{ij}} \right).
\end{align*}
    A \emph{local vector} 
     is a vector-valued function on weighted
    graphs $G$ that can be expressed as a finite linear combination of
    $\btt(X,G)$ for some rooted test graphs $X$.
\end{defn}

In particular, the usual weighted homomorphism density $t(X,G)$ from
Definition~\ref{defn:leftconv} for the unrooted version of a rooted
test graph $X$ is the average of the entries of $\btt(X,G)$. The
utility of Definition~\ref{defn:rootedtest} is that while we are
ultimately interested in the scalar quantity $\E[
H(\bpsi^{(K)})]/N$, this quantity is defined by the vector recursion
\eqref{eqn:psirecursion}. With the notion of rooted test graphs,
\eqref{eqn:psirecursion} and the evaluation of $
H(\bpsi^{(K)})/N$ can be encoded by graphical operations.

\begin{lem}
    For any $\eps,t,\delta>0$, and an $\alpha$-good potential $V$,
     there exists a local function
    $q_{\eps}:\cG_{[0,C],C}\times[0,1]\rightarrow\R^N$ which can be written
    as a finite sum
    \begin{align*}
    \qq_{\eps}(G;\beta) = \sum_{X\in\cS} c_X(\beta) \btt(X,G)\, ,
    \end{align*}
    such that $\cS$ is a finite set, the coefficients $c_X(\beta)$
    are  polynomials in $\beta$ independent of $G$, and 
    for all $G\in\cG_{[0,C],C}$ and $\beta\in[0,1]$, we have (for $K=t/\delta$)
    \begin{align*}
    \left| \frac{1}{N}\E[H(\bpsi^{(K)})] - q_{\eps} (G;\beta) \right| \le \eps.
    \end{align*}
\end{lem}
\begin{proof}
We will decorate our $G$ with $K+1$ extra external fields
$\bpsi^{(0)}$ and $\bet^{(0)},\bet^{(1)}\ldots,\bet^{(K-1)}$. These
will first be regarded as constants so we can express $\frac1N
H(\bpsi^{(K)})$ as a linear combination of
$t(X,(\BB,\hh,\bpsi^{(0)},\bet^{(0)},\ldots, \bet^{(K-1)}))$ for some
test graphs $X$, and then we will take expectations to recover a
linear combination of $t(X,(\BB,\hh))$.

Note that $\hh,\bpsi^{(0)},\bet^{(0)},\ldots,\bet^{(K-1)}$ are all
local vectors by considering $t(X,G)$ for a test graph $X$ with one
vertex whose weight vector in $\Z_{\ge0}^{K+2}$ is in the standard
basis. By definition, linear combinations of local vectors are also
local vectors. Thus, if we show that the set of local vectors is
closed under componentwise products and multiplication by $\BB$, then
we may conclude from \eqref{eqn:psirecursion} that $\bpsi^{(K)}$ is a
local vector. This is as each $\PP^{(k)}$ is a fixed componentwise
polynomial function of $\hh$ and $\bpsi^{(k)}$.

By linearity, to show closure under componentwise multiplication, it
suffices to prove that the componentwise product of $\btt(X_1,G)$ and
$\btt(X_2,G)$ is local for all rooted test graphs $X_1$ and $X_2$. In
fact, this componentwise product is equal to $\btt(X,G)$ for $X$ the
rooted test graph given by identifying the root vertices of $X_1$ and
$X_2$ and summing the weight vectors at the root. This is clear from
the expression for $\btt(X,G)$ in Definition~\ref{defn:rootedtest}.

Similarly, to show closure under multiplication by $\BB$, it suffices
to prove that $\BB\btt(X,G)$ is a local vector for all rooted test
graphs $X$. In fact, this product is equal to $\btt(Y,G)$ for $Y$ the
rooted test graph given by adding one new vertex $v$ with zero weight
to $X$, connecting it to the root of $X$ with a single edge, and
declaring $v$ as the new root of $Y$. This is also clear from the
expression for $\btt(X,G)$ in
Definition~\ref{defn:rootedtest}.

Now, given two rooted test graphs $X_1$ and $X_2$, consider the
unrooted test graph $X$ given by connecting the roots of $X_1$ and
$X_2$ with a single edge. By Definition~\ref{defn:leftconv} and
Definition~\ref{defn:rootedtest}, it is clear that we have the
identity
\begin{align*}
t(X,G) = \frac1N \<\btt(X_1,G), \BB\btt(X_2,G)\>.
\end{align*}
In particular, summing this over the linear combination of weighted
homomorphism density vectors for $\bpsi^{(K)}$ and recalling that
$H(\bphi)=\frac12\<\bphi, \BB\bphi\>$, we may conclude that $
H(\bpsi^{(K)})/N$ is a linear combination of $t(X,G)$ for some unrooted
test graphs $X$. By \eqref{eqn:psirecursion}, the coefficients of this
linear combination are polynomials in $\beta$ and $\delta^{1/2}$ with
no dependence on $G$.

Finally, we will take expectations with respect to
$\bpsi^{(0)},\bet^{(0)},\bet^{(1)}\ldots, \bet^{(K-1)}$. For this
part, we will instead write $H(\bpsi^{(K)})/N$ as a linear
combination of $t_\inj(X,G)$ using the inclusion-exclusion argument as
in the remarks after \eqref{eqn:inj}. This is permissible as all constructed graphs are connected. When this is done, taking
expectations with respect to
$\bet^{(0)},\bet^{(1)}\ldots,\bet^{(K-1)}$ amounts to substituting in
Gaussian moments into the terms $(\eta^{(\ell)}_{f(i)})^{d_{\ell,i}}$
in \eqref{eqn:inj} because we have restricted the sum to injections.

Taking expectations with respect to $\bpsi^{(0)}$ is a bit more subtle
because the law of $(\psi^{(0)}_{f(i)})^{d_i}$ depends on $i$.
Nevertheless, because $\mu^{(0)}_\beta=\mu_0$ is a product measure and
we are only summing over injections $f$, we may substitute each
$(\psi^{(0)}_{f(i)})^{d_i}$ for $M_{d_i}(h_i)$ when taking
expectations with respect to $\bpsi^{(0)}$ in \eqref{eqn:inj}, where
$M_d(h)=\int_\R \varphi^d e^{-V_h(\varphi)} \de\varphi$ is the $d$th
moment of the renormalized single-spin measure. By our regularity
assumptions on $V$, $M_d$ is a continuous function on $[-C,C]$ for all
$d$. Hence, they may be approximated by polynomials to arbitrarily
small error. Replacing the $M_{d_i}(h_i)$ with these polynomials of
$h_i$ in the expectation, we obtain a linear combination of certain
injective homomorphism densities $t_\inj(X,G)$ whose coefficients are
polynomials in $\beta$ and $\delta^{1/2}$. Since $\beta$ lies in the
compact interval $[0,1]$, the resulting linear combinations converge
to $ \E[H(\bpsi^{(K)})]/N$ uniformly in
$[0,1]\times\cG_{[0,C],C}$ as the approximation errors to
$M_{d_i}(h_i)$ converge to $0$. This is as we start with a fixed
expression for $H(\bpsi^{(K)})/N$ before taking expectations so
there is some bound on the exponents $d_i$ needed as well as the
number and size of terms in which they appear.
\end{proof}

\begin{proof}[Proof of Proposition~%
\ref{prop:locenergy}] The proof is already finished, but we will
summarize the steps for completeness and clarity. What we have just
shown is that $\E[H(\bpsi^{(K)})]/N$ is a
$[0,1]\times\cG_{[0,C],C}$-uniform limit of linear combinations of
injective homomorphism densities $t_\inj(X,G)$, hence ordinary
homomorphism densities $t(X,G)$, whose coefficients are polynomials in
$\beta$ (and $\delta^{1/2}$). Thus, $\beta\mapsto \frac1N
\E[H(\bpsi^{(K)})]$ is a norm-continuous map from $[0,1]$ to
$\cL(\cG_{[0,C],C})$. This implies, by our inductive argument in
Lemma~\ref{lem:afterbackprop}, the same for $\beta\mapsto 
\E[H(\btau^{(K)})]/N$, as well as for $\beta\mapsto \langle
H(\bphin)\rangle_\beta^{(t)}/N$, by our approximation argument after
Proposition~\ref{prop:discretize}.
\end{proof}

\section{Discussion}\label{sec:future}

The applications of Theorem~\ref{thm:main} in Section~\ref{sec:Disc} 
can be interpreted as providing a ``derandomization'' result for spin glasses on 
sparse regular graphs at high temperature. 
Indeed, the free energy asymptotics in Eq.~\eqref{eq:SpinGlassFreeEnergy} is 
standard when $G_N$ is a random $k$-regular graph (with $A_{ij}\sim\Unif(\{+\beta,-\beta\})$,
for $(i,j)\in E_N$). 
 Theorem~\ref{thm:main} allows for the transfer of this result to deterministic or pseudorandom 
 instances having the same left limit and satisfying the spectral condition.
 Similar considerations apply to Corollaries \ref{cor:Constrained} and \ref{cor:Antiferromagnetic}.

Remarkably, all of these models have the same asymptotics of the free energy, even if
the $\{A_{ij}:(i,j)\in E_N\}$ are non-random.  This is ultimately related to the fact that the local 
weak limit of the corresponding graphs (the rooted $k$-regular tree) is bipartite.

It would be interesting to find analogues of the previous results for graphs
where a spectral  condition does not apply but fast mixing for dynamics holds, 
see for example \cite{LMRW}.

An extension of Theorem~\ref{thm:main} to dense disordered models would 
also be interesting. This amounts to the removal of the assumption (SG).
In our proof, this assumption to control the difference between the energy computed by discretized Langevin dynamics and by the polynomial approximation to those dynamics. 
We expect this to be a technical artifact of our proof and that locality of the free energy should hold for dense models as well. 

Finally, it would also be interesting to establish similar theorems not only showing convergence of Ising model free energies but also of the Ising measure itself, in the local weak sense. 
For instance, in the examples of Section \ref{sec:coro}
we expect that, under the spectral condition, the limit should be the free boundary Gibbs measure on the Benjamini-Schramm limit of the underlying weighted graph.

\section*{Statement of AI use}

GPT-5.5 was used to in carrying out 
the analysis of functions $\phi_2(q;\beta)$ and $\psi(q;\beta,\gamma)$
in Appendices \ref{sec:proof-spin-glass} and \ref{sec:proof-free-energy-random-ferro}, respectively. GPT-5.6 was then used to check for mathematical and typographical issues upon completion of the manuscript. All AI-assisted calculations and suggestions
were independently checked by the authors, who take full responsibility for the contents of
the paper.

\section*{Acknowledgments}
AM was partially supported by the NSF Award MFAI-2501597.
MR was supported by the National Defense Science and Engineering Graduate (NDSEG) Fellowship Program. This material is based upon work supported by the Air Force Office of Scientific Research under award number FA9550-25-C-B010 in the amount of \$43,200.

\appendix

\section{Some technical lemmas for the proof of Theorem \ref{thm:main}}\label{app:lemmas}

In this appendix, we give proofs of some technical results that were
previously skipped. 

\subsection{Proof of Lemma~\ref{lem:bregman}}\label{subapp:bregman}

    Note that the hypotheses and conclusion do not change if an affine
    function is added to $f$ or if the domain is shifted or reflected.
    Thus, we may without loss of generality assume that $b=0$,
    $a\ge0$, and $f(b)=f'(b)=0$, which means that $f'(a)>U$ and
    $f''(x)\le W$ for $\max(a-1,0)\le x\le a$. If $a\le U/W\le1$, then
    $f''(x)\le W$ for $x\in[0,a]$ so $f'(a)\le Wa\le U$, a
    contradiction. Hence, we must have that $a>U/W$, which means that
    the bound $f''(x)\le W$ holds for all $x\in[a-U/W,a]$. In
    particular, we have on this interval that $f'(x)\ge
    f'(a)-W(a-x)>U-W(a-x)$. Since $f$ is convex with $f'(0)=0$, it is
    increasing on $[0,a]$. We may conclude that \begin{align*} f(a)\ge
    f(a)-f(a-U/W)=\int_{a-U/W}^af'(x)\de x>\int_{a-U/W}
    ^a(U-W(a-x))\de x=\frac{U^2}{2W},
\end{align*}
as desired.

\subsection{Proof of Proposition~\ref{prop:discretize}}\label{subapp:discretize}

    By scaling time, we may without loss of generality assume that $\sigma = 1$. We have by the SDE that
    \begin{align*}
\xx^{(\delta(n+1))}=\xx^{(\delta n)}+\int_{\delta
n}^{\delta(n+1)}a(\xx^{(t)})\de t+(\ww^{(\delta(n+1)} -\ww^{(\delta n)}),
\end{align*}
    which means that
    \begin{align*}
        \xx^{(\delta(n+1))} - \yy^{(n+1)} &= \xx^{(\delta n)} -
        \yy^{(n)} + \int_{\delta n}^{\delta(n+1)} (a(\xx^{(t)}) -
        a(\yy^{(n)})) \de t \\ &= \xx^{(\delta n)} - \yy^{(n)} +
        \delta(a(\xx^{(\delta n)}) - a(\yy^{(n)})) + \int_{\delta
        n}^{\delta(n+1)} (a(\xx^{(t)}) - a(\xx^{(\delta n)}))\de t.
    \end{align*}

    Taking the $L^2$ norm and using the Lipschitz property of $a$
    gives
    \begin{align}
        \E[ \|\xx^{(\delta(n+1))}-\yy^{(n+1)}\|^2 ]^{1/2} &\le
        (1+\delta K_3) \E[ \|\xx^{(\delta n)}-\yy^{(n)}\|^2 ]^{1/2} +
        K_3\int_{\delta n}^{\delta(n+1)} \E[ \|\xx^{(t)}-\xx^{(\delta
        n)}\|^2 ]^{1/2}\de t
        \label{eqn:x-yrecursion} \\
        &\le (1+\delta K_3) \E[ \|\xx^{(\delta n)}-\yy^{(n)}\|^2
        ]^{1/2} + K_3\delta^{1/2} \left( \int_{\delta n}^{\delta(n+1)}
        \E[ \|\xx^{(t)}-\xx^{(\delta n)}\|^2 ]\de t \right)^{1/2}.
        \nonumber
    \end{align}

    On the other hand, for any times $s\le t$ we may estimate
    \begin{align}
        \E[ \|\xx^{(t)} - \xx^{(s)}\|^2 ] &= \E\left[ \left\|\int_s^t
        a(\xx^{(u)})\de u + (\ww^{(t)} - \ww^{(s)})\right\|^2 \right]
        \label{eqn:xt-xs} \\ &\le 3\E\left[ \left\|\int_s^t
        (a(\xx^{(u)}) - a(\bzero)) \de u\right\|^2 \right] + 3\E\left[
        \left\|\int_s^t a(\bzero) \de u\right\|^2 \right] +
        3\E[\|\ww^{(t)} - \ww^{(s)}\|^2 ] \nonumber \\ &\le
        3(t-s)\E\left[ \int_s^t \|a(\xx^{(u)}) - a(\bzero)\|^2 \de u
        \right] + 3N(1+K_2)(t-s) \nonumber \\ &\le 3(t-s)K_3^2
        \int_s^t \E[ \|\xx^{(u)}\|^2 ]\de u + 3N(1+K_2)(t-s). \nonumber
    \end{align}

    Finally, It\^o's lemma gives
    \begin{align*}
 \|\xx^{(u)}\|^2 = \|\xx^{(0)}\|^2 + 2\int_0^u\<\xx^{(t)},
 a(\xx^{(t)})\>\de t + 2\int_0^u\<\xx^{(t)}, \de\ww^{(t)}\> + Nu.
\end{align*}
    By assumption, $\xx^{(t)}$ is the strong solution to an SDE with
    globally Lipschitz coefficients, so $\int_0^u\<\xx^{(t)},
    \de\ww^{(t)}\>$ is a martingale. Using the
    bound
    \begin{align*}
|\<\xx, a(\xx)\>| \le \|\xx\|\|a(\xx)\| \le \|\xx\|(\|a(\bzero)\| +
K_3\|\xx\|) \le \|a(\bzero)\|^2 + (1+K_3)\|\xx\|^2 \le
K_2N+(1+K_3)\|\xx\|^2,
\end{align*}
    we may take expectations to obtain
    \begin{align*}
        \E[ \|\xx^{(u)}\|^2 ] &= \E[ \|\xx^{(0)}\|^2 ] +
        2\int_0^u\E[\<\xx^{(t)}, a(\xx^{(t)})\>]\de t + Nu \\ &\le
        K_1N+(1+2K_2)Nu+2(1+K_3)\int_0^u\E[\| \xx^{(t)}\|^2]\de t.
    \end{align*}

    The function $f(u)=\frac1N\E[\|\xx^{(u)}\|^2]+\frac{1+2K_2}
    {2(1+K_3)}$ then satisfies
    \begin{align*}
f(u)\le K_1+\frac{1+2K_2}{2(1+K_3)}+2(1+K_3)\int_0^uf(t) \de t,
\end{align*}
    so Gr\"onwall's inequality implies the estimate
    $\E[\|\xx^{(u)}\|^2]\le K_4(1+e^{K_4u})N$ for some positive
    constant $K_4$ depending only on $K_1,K_2,K_3$. Hence for times
    $s\le t\le T$, we have from \eqref{eqn:xt-xs} that
    \begin{align*}
\E[ \|\xx^{(t)}-\xx^{(s)}\|^2 ] \le
(3TK_3^2K_4(1+e^{K_4T})+3(1+K_2))(t-s)N \le K_5(1+Te^{K_5T})(t-s)N
\end{align*}
    for some positive constant $K_5$ depending only on $K_1,K_2,K_3$.
    In particular, we have for $n<T/\delta$ that
    \begin{align*}
\int_{\delta n}^{\delta(n+1)} \E[ \|\xx^{(t)}-\xx^{(\delta n)}\|^2 ]\de t
\le K_5(1+Te^{K_5T})\delta^2N.
\end{align*}

    Denoting $D_n=N^{-1/2}\E[\|\xx^{(\delta n)}-\yy^{(n)}\|^2]^{1/2}$,
    we can plug this estimate into \eqref{eqn:x-yrecursion} to obtain
    for all $n<T/\delta$ the bound $D_{n+1}\le (1+\delta
    K_3)D_n+M\delta^{3/2}$ where $M=K_6(1+Te^{K_6T})$ for some positive
    constant $K_6$ depending only on $K_1,K_2,K_3$. Since $D_0=0$, we
    deduce the bound \begin{align*} D_{T/\delta}\le
    M\delta^{3/2}\sum_{k=0}^{T/\delta-1}(1+\delta
    K_3)^k\le MTe^{K_3T}\delta^{1/2}\le e^{K_7(T+1)}\delta^{1/2}
\end{align*}
for some constant $K_7$ depending only on $K_3$ and $K_6$. Taking $L = 2K_7+e^{2K_7}$ gives the desired conclusion.
%
%
\subsection{Proof of Proposition~\ref{prop:wtapprox}}\label{subapp:wtapprox}

    For $-C\le h_1\le\cdots\le h_m\le C$, consider a smooth partition
    of unity $g_1(h)+\cdots+g_m(h)=1$ of $[-C,C]$ such that each
    $\supp(g_i)$ is a closed interval with $h_i\in\supp(g_i)$ and
    $h_j\notin\supp(g_i)$ for $i\ne j$. In other words,
    $\supp(g_i)\subset(h_{i-1},h_{i+1})$, where $h_0=-\infty$ and
    $h_{m+1}=\infty$.

    For each fixed $h$, there exists by standard weighted
    approximation theory, see for example \cite{lubinsky2007survey}, a polynomial
    $P_h\in\R[\varphi]$ with
    \begin{align*}
\sup_{\varphi\in\R} e^{-A\varphi^2} |V_h'(\varphi)-P_h(\varphi)| <
\frac\epsilon{3m}.
\end{align*}
    By the Stone-Weierstrass theorem, there exist polynomials
    $Q_i\in\R[h]$ such that
    \begin{align*}
\sup_{h\in[-C,C]}|Q_i(h)-g_i(h)| <
\frac\epsilon{3m\cdot\sup_{\varphi\in\R} e^{-A\varphi^2}
|P_{h_i}(\varphi)|}
\end{align*}
    for each $i$.

    We consider the approximation
    $P(h,\varphi)=\sum_{i=1}^mP_{h_i}(\varphi) Q_i(h)$ to
    $V_h'(\varphi)$. By construction, we have for $h\in[h_i,h_{i+1}]$
    that
    \begin{align*}
        |V_h'(\varphi)-P(h,\varphi)| &\le \left|V_h'(\varphi) -
        \sum_{i=1}^m V_{h_i}'(\varphi)g_i(h) \right| +
        \left|\sum_{i=1}^m (V_{h_i}' - P_{h_i}(\varphi)) g_i(h)
        \right| + \left| \sum_{i=1}^m P_{h_i}(\varphi)(g_i(h) -
        Q_i(h)) \right| \\ &\le L(h_{i+1}-h_i) + \frac\epsilon3 +
        \frac\epsilon{3m} \sum_{i=1}^m \frac{|P_{h_i}(\varphi)|}
        {\sup_{\varphi\in\R} e^{-A\varphi^2} |P_{h_i}(\varphi)|},
    \end{align*}
    where $L$ is an upper bound on the Lipschitz constant of
    $V_h'(\varphi)$ in $h$ that is uniform over all $\varphi\in\R$.
    Such an $L$ exists by our definition of an $\alpha$-good family of
    potentials $V$ from Theorem~\ref{defn:contFV}. Choosing
    $h_1\le\cdots\le h_m$ so that $L(h_{i+1}-h_i)<\frac\epsilon3$ for
    all $i$ yields the desired bound
    \begin{align*}
e^{-A\varphi^2} |V_h'(\varphi)-P(h,\varphi)| < \epsilon
\end{align*}
    for all $(h,\varphi)\in[-C,C]\times\R$.

\section{Proofs of corollaries from Section
\ref{sec:Disc}}
\label{app:CoroProofs}

The following lemma is standard, and will be used
repeatedly in the proof of statements from Section \ref{sec:Disc}.
\begin{lem}
\label{lemma:left-convergence-locally-tree-like}
Let $G_N=(\AA_N,\hh_N)$ be a sequence of weighted $k$-regular graphs
which converge locally to the $k$-regular tree. Let the edge weights
$(A_{ij}:\{i,j\}\in E(G_N))$ be i.i.d. random variables with common
distribution $Q_A$, and independent of the vertex weights $(h_{i}:i\in
[N])$ which are i.i.d. random variables with common distribution
$Q_h$. Assume that $Q_A,Q_h$ are independent of $N$ and have bounded
support. For a constant $c\in\RR$, define
$G_N(c)=(\AA_N(c),\hh_N(c))$, where
$\AA(c)=\AA_N-c\bone\bone^{\sT}/N$, $\hh_N(c)=\hh_N$.

Then, the sequence $(G_N(c): N\ge 1)$ is almost surely
left-convergent.
\end{lem}
\begin{proof}
First of all, as remarked in the introduction, it is sufficient to
prove convergence of  homomorphism counts when $X$ is a connected graph.
Fixing such a graph $X$, we can assume that $G_N$ has
girth larger than $2|V(X)|+1$. If this is not the case,
we can always modify $G_N$ for $o(N)$ edges, and achieve the desired girth, 
and this does not modify the limit of $t(X,G_N(c))$.

Given a graph $X=(\WW,\dd)$, 
we let $E(\WW)$ be the set of edges, where vertices $i$, $j$ 
are connected by exactly $w_{i,j}$ distinct edges. 
For $\ell:E(\WW)\to \{0,1\}$, we let $|\ell|=\sum_{e\in E(\WW)} \ell(e)$,
and $\cF(X,\ell)$ be the set of connected components of the graph 
obtained by removing the edges of $X$ with $\ell(e)=1$. 
We distribute the vertex weights $d_i = \sum_{Z\in \cF(X,\ell)} d_{Z,i}$
arbitrarily.
We then have
\begin{align*}
t(X,G_N(c)) &= \frac{1}{N}\sum_{\ell\in\{0,1\}^{E(\bW)}}
\sum_{f:[n]\to [N]} 
\left( \prod_{i\in[n]}  h_{f(i)}^{d_{i}}
\right) \left( \prod_{i\le j\in[n]} 
\Big(a_{f(i)f(j)}-\frac{c}{N}\Big)^{w_{ij}} \right)\\
& = \frac{1}{N}
\sum_{\ell\in \{0,1\}^{E(X)}} \left(-\frac{c}{N}\right)^{|\ell|}
\prod_{Z\in \cF(X,\ell)}
\sum_{f_Z:V(Z)\to [N]}
\left( \prod_{i\in V(Z)}  h_{f_Z(i)}^{d_{Z,i}}
\right) \left( \prod_{i\le j\in V(Z)} 
a_{f_Z(i)f_Z(j)}^{w_{Z,ij}} \right)\\
&= \frac{1}{N}
\sum_{\ell\in \{0,1\}^{E(X)}} \left(-\frac{c}{N}\right)^{|\ell|}
N^{|\cF(X,\ell)|}\prod_{Z\in \cF(X,\ell)} t(Z,G_N)
\end{align*}
%
It is easy to see that $|\cF(X,\ell)|\le |\ell|+1$ and equality 
holds only if all edges $e$ with $\ell(e)=1$ are cut edges of $X$. If equality does not hold, then the contribution of $\ell$ to the sum is $O(N^{-1})$ and vanishes in the limit. Otherwise, the contribution of $\ell$ to the sum is
\[
    (-c)^{|\ell|}\prod_{Z\in\cF(X,\ell)}t(Z,G_N).
\]
Since the sum runs over a bounded number of terms and each of the $t(Z,G_N)$
converges, we obtain that $t(X,G_N(c))$ converges.
\end{proof}

We pause to emphasize two immediate consequences of Lemma
\ref{lemma:left-convergence-locally-tree-like}.
\begin{rem}
For $Q_A=\delta_{a}$, a Dirac mass at $a$, and $\AA_N$, $\hh_N$ as in
Lemma \ref{lemma:left-convergence-locally-tree-like}, define
$G^{\perp}_N=(\PP^{\perp}\AA_N\PP^{\perp},\hh_N)$. Since
$\PP^{\perp}\AA_N\PP^{\perp}=\AA_N(c)$ for $c=ka$, it follows that
$(G_N^{\perp}: N\ge 1)$ is almost surely left-convergent.
\end{rem}

\begin{rem}
Let $(G^{(1)}_N(a):N\ge 1)$, $(G^{(2)}_N(a):N\ge 1)$ be two sequences
as per Lemma \ref{lemma:left-convergence-locally-tree-like} (with the
same weights distributions $Q_A$, $Q_h$). Then, these are left-convergent to the same limit, as can be seen by considering any
sequence obtained by interleaving the $(G^{(1)}_N(a))$ and
$(G^{(2)}_N(a))$.
\end{rem}

By taking one of these two sequences in the last remark to be a random
$k$-regular graph, and applying Theorem \ref{thm:main}, we obtain the
following.
\begin{cor}\label{coro:ReductionToRandomGraphs}
Let $(G_N=(\AA_N,\hh_N):N\ge 1)$ be a sequence of weighted $k$-regular
graphs satisfying the assumptions of Lemma
\ref{lemma:left-convergence-locally-tree-like}, and let
$G^{(r)}_N=(\AA^{(r)}_N,\hh^{(r)}_N)$ be a random $k$-regular graph
(either uniformly random or distributed according to the configuration
model) with the same edge-weight and vertex-weight distributions as $G_N$.


Fix $c\in\RR$. If (eventually almost surely) it holds that
$\lambda_{\max}(\AA_N(c)),\lambda_{\min}(\AA_N(c))\in I$ and $\lambda_{\max}(\AA^{(r)}_N(c)),\lambda_{\min} (\AA^{(r)}_N(c))\in I$, for some $N$-independent interval $I$ of length less than $1$ then, almost
surely,
\begin{align}
\lim_{N\to\infty} F(G_N(c))&= \lim_{N\to\infty} F(G^{(r)}_N(c))\, .
\end{align}
\end{cor}
In view of this result, in order to establish Corollaries
\ref{cor:spin-glass}-\ref{cor:Antiferromagnetic}, it will be
sufficient to check the spectral condition and work with random
$k$-regular graphs.

\subsection{Proof of Corollary
\ref{cor:spin-glass}}
\label{sec:proof-spin-glass}


We note that the spectral condition
$|\lambda_{\max}(\AA_N)|,|\lambda_{\min}(\AA_N)|\le\frac12-\delta$ holds by
\cite[Theorem 1.2]{mohanty2021explicit}, while
$|\lambda_{\max}(\AA_N^{(r)})|,|\lambda_{\min}(\AA_N^{(r)})|\le \frac12-\delta$
holds by \cite[Theorem 1.16]{mohanty2021explicit}.

 Applying Corollary \ref{coro:ReductionToRandomGraphs}, we can reduce
 to the case of random $k$-regular graphs, which can be
 treated by a second moment calculation. This derivation is standard,
 but we could not find a citable reference (the closest being
 \cite{montanari2006rigorous}), so we give a self-contained proof.

 Dropping for simplicity the superscript $(r)$, let
 $G_N=(\AA_N,\bzero)$ be a random $k$-regular graph on vertex set
 $[N]$, distributed according to the configuration model, with
 independent edge weights $A_{ij}\sim\Unif(\{-\beta,\beta\})$. Write
 $Z_N:=Z(G_N)$ and $L:=kN/2$ for the number of edges, and
\begin{align}\label{eq:PhiBetaDef}
    \phi(\beta) = \log 2+\frac{k}{2}\log\cosh\beta\, .
\end{align}
Computing the first moment is straightforward since
$\E\{e^{H(\bsigma)}\}=(\cosh\beta)^L$ for any $\bsigma\in\{+1,-1\}^N$,
whence
 \begin{align}\label{eq:spin-glass-first-moment}
 \E[Z_N] &= 2^N(\cosh\beta)^L= e^{N\phi(\beta)}\, .
 \end{align}
 For the second moment, it is useful to define the constrained
 partition function
 \begin{align}
    Z_N(q;\beta) :=\sum_{\substack{\bsigma_1,\bsigma_2\\
    d_H(\bsigma_1,\bsigma_2)=N-m}} e^{H(\bsigma_1)+H(\bsigma_2)}\, ,
 \end{align}
 $m=N(1+q)/2$ and $d_H(\,\cdot\, ,\,\cdot\, )$ denotes the Hamming
 distance.

The expectation of the constrained partition function is slightly more
involved. Denoting by $\Pair(2n):=(2n)!/(n!2^n)$ the number of
pairings of $2n$ objects, and assuming without loss of generality
$q\ge 0$ (since the result is symmetric in $q$), we have, with
$\om\equiv N-m$,
\begin{equation}
\label{eq:spin-glass-constrained-expectation}
\E\,Z_N(q;\beta) = \binom{N}{m}\frac{2^N}{\Pair(Nk)}
\tilde{\sum}_{\ell=0}^{\om k}\binom{mk}{\ell}\binom{\om k}{\ell} \ell!
\, \Pair(mk-\ell)\Pair(\om k-\ell) (\cosh 2\beta)^{L-\ell} \, .
\end{equation}
Here the sum $\tilde{\sum}$ is restricted to the values of $\ell$ such
that $mk-\ell$ is even. In this expression: $2^N\binom{N}{m}$ is the
number of configurations of the spins at Hamming distance $m$;
$1/\Pair(Nk)$ is the number of the probability of any specific graph
(in the configuration model); the index $\ell$ counts the number of
edges that connect a vertex $i$ such that $\sigma_{1,i}=\sigma_{2,i}$
to a vertex $j$ such that $\sigma_{1,j}\neq\sigma_{2,j}$;
$\binom{mk}{\ell}\binom{\om k}{\ell}$ is the number of ways of
choosing the half edges; $\ell! \, \Pair(mk-\ell)\Pair(\om k-\ell)$ is
the number of pairings; finally $(\cosh 2\beta)^{L-\ell}$ is the
expectation of $\exp(H(\bsigma_1)+H(\bsigma_2))$ over the edge signs.

After some straightforward algebra, we obtain
\begin{align}
    \E\,Z_N(q) &= 2^N\binom{N}{m} L!\binom{Nk}{mk}^{-1}
    \sum_{\ell=0}^{\om k}\frac{1}{\ell!}\frac{1}{((mk-\ell)/2)!}
    \frac{1}{((\om k-\ell)/2)!} 2^{\ell} (\cosh 2\beta)^{L-\ell} \\ &=
    2^N\binom{N}{m} \binom{Nk}{mk}^{-1} M(h,\beta)^{\frac{Nk}{2}}
    e^{-hmk} \P\left(\sum_{i=1}^{Nk/2}Y_i= mk\right)\, ,
\end{align}
where the last equality holds for any $h\ge0$, $M(h;\beta) =
(1+e^{2h})\cosh 2\beta +2e^h$, and the $Y_i$'s are i.i.d. random
variables with common distribution
\begin{align}
    \P(Y_i=0) = \frac{\cosh 2\beta}{M(h,\beta)} \, ,\;\;\;\;\;
    \P(Y_i=1) = \frac{2e^h}{M(h,\beta)}\, ,\;\;\;\;\; \P(Y_i=2) =
    \frac{e^{2h}\cosh 2\beta}{M(h,\beta)}\, .
\end{align}

Let us now choose $h=h_q\ge0$ so that $\E Y_i =\frac{2m}N  = 1+q$, which is possible under our assumption that $q\ge0\implies m\ge\frac N2$. The value of $h_q$ is determined as the unique nonnegative solution of the equation
\begin{equation}
1+q = \frac{M'(h_q;\beta)}{M(h_q;\beta)} =: 1+Q(h_q;\beta),
\end{equation}

By an application of the local central limit theorem for lattice random
variables, we obtain, for some constant $C>0$,
\begin{align}
\E\,Z_N(q) &= C\, N^{-1/2} \, e^{N\phi_2(q;\beta)}
\big(1+o_N(1)\big)\, ,
\end{align}
where letting $\entro(x):=-x\log x-(1-x)\log(1-x)$, we have
\begin{align}
\phi_2(q;\beta):= \log 2 -(k-1) \entro\!\left(\frac{1-q}{2}\right)
+\frac{k}{2} \log M(h_q,\beta)
-k\frac{1+q}{2}h_q.\label{eq:spin-glass-phi2-def}
\end{align}
%
It is easy to see that $\phi_2(0;\beta)=2\phi(\beta)$, and
$\phi_2(-q;\beta)=\phi_2(q;\beta)$. By the second moment method (see
e.g. \cite{montanari2006rigorous,montanari2024friendly}), we obtain
\begin{align}
    \phi_2(q;\beta)<\phi_2(0;\beta)\;\;\forall q\in (0,1] \;\;
    \Rightarrow\;\; \lim_{N\to\infty}\frac{1}{N}\log Z_N=\phi(\beta)\,
    ,
\end{align}
The proof is completed by showing that the condition
$\phi(q;\beta)<\phi(0;\beta)$ is satisfied provided we have
$(k-1)(\tanh\beta)^2<1$.

To prove this claim, note that, after some simple algebra
\begin{align}
Q(h;\beta) = \frac{c\sinh h}{1+c\cosh h}\, ,
\end{align}
where $c:=\cosh 2\beta$. We write for short $Q(h):=Q(h;\beta)$. The
function $h\mapsto Q(h)$ is strictly increasing from $0$ to $1$ as $h$
ranges from $0$ to $\infty$, and satisfies
\begin{equation}\label{eq:spin-glass-qh}
    \frac{1+Q(h)}{1-Q(h)}=\frac{1+c\,e^h}{1+c\, e^{-h}}\, .
\end{equation}
The envelope relation $1+q=M'(h;\beta)/M(h;\beta)$ and the
definition of $\phi_2$ give
\begin{equation}\label{eq:spin-glass-dphi2dq}
    \frac{\de }{\de q}\phi_2(q;\beta)
    =\frac{k-1}{2}\,\log\frac{1+q}{1-q}-\frac{k} {2}\,h_q\, .
\end{equation}
Since the map $h\mapsto Q(h)$ is a monotone bijection, we have
$\phi_2'(q)>0$ for some $q>0$ if and only if, for $h=Q^{-1}(q)$,
\begin{equation}\label{eq:spin-glass-f}
    f(h):=kh-(k-1)\log\frac{1+ce^h}{1+ce^{-h}}<0\, .
\end{equation}
We will show that instead $f(h)>0$ for all $h>0$. Differentiating
\eqref{eq:spin-glass-f} gives
\begin{align*}
    f'(h)&=k-(k-1)\frac{2c\,(c+\cosh h)}{1+c^2+2c\cosh h}\\
    &=k-2(k-1)c\cdot A(\cosh h)\, ,
\end{align*}
where $A(x):=(c+x)/(1+c^2+2cx)$. Since $A(x)$ is strictly decreasing
on $[1,\infty)$, $f'(h)$ is strictly increasing on $[0,\infty)$.
Therefore $f(h)>0$ for all $h>0$ as soon as $f'(0)>0$, and
\begin{align}
    f'(0)>0 &\Longleftrightarrow k(1+c)>2c(k-1) \Longleftrightarrow
    c<\frac{k}{k-2}\notag\\ &\Longleftrightarrow
    (k-1)(\tanh\beta)^2<1\, ,
\end{align}
using $c=\cosh 2\beta$. We conclude that
$\phi_2(q;\beta)<\phi_2(0;\beta)$ for every $q\in(0,1]$ as claimed.

\subsection{Proof of Corollary
\ref{cor:Constrained}}

With a slight abuse of notation, we define a generalized balanced
partition function and free energy
\begin{align}
    Z_{b,N}(c;\eps) = \sum_{\substack{\bsigma\in\{-1,1\}^N\\
    |\<\bsigma,\bone\>|\le N\eps}} e^{\<\bsigma,
    \AA_N(c)\bsigma\>/2}\, , \;\;\;\;\; F_{b,N}(c;\eps) =
    \frac{1}{N}\log Z_{b,N}(c;\eps)\, ,
\end{align}
as well as their unbalanced counterparts
\begin{align}
Z_{u,N}(c;\eps) = \sum_{\substack{\bsigma\in\{-1,1\}^N\\
|\<\bsigma,\bone\>|> N\eps}} e^{\<\bsigma, \AA_N(c)\bsigma\>/2}\, ,
\;\;\;\;\; F_{u,N}(c;\eps) = \frac{1}{N}\log Z_{u,N}(c;\eps)\, .\label{eq:ZuDef}
\end{align}
The overall free energy and partition functions are then
\begin{align}
Z_N(c) = Z_{b,N}(c;\eps) + Z_{u,N}(c;\eps)\, , \;\;\;\;\; F_N(c) =
\frac{1}{N}\log Z_N(c) \, .
\end{align}

We recall the definition of $\phi(\beta)$ from
Eq.~\eqref{eq:PhiBetaDef}, and state two lemmas which will be proved
in Sections \ref{sec:proof-free-energy-random-ferro} and
\ref{sec:proof-comparison-bounded-unbounded}, respectively.
\begin{lem}\label{lemma:free-energy-random-ferro}
Let $G_N^{(r)}=(\AA_N^{(r)},\bzero)$ be a sequence of random
$k$-regular graphs ($k\ge 3$) distributed according to the
configuration model, with constant edge weight $\beta\in\R$, and let
$F^{(r)}_N(\gamma)$ be the free energy density of the shifted graph
$G_N^{(r)}(\gamma)$. If $|\beta|<1/(4\sqrt{k-1})$, then there exists
$\eta=\eta(\beta,k)>0$ such that
\begin{align}
|\gamma-k\beta|<\eta\;\; \Rightarrow\;\;
\lim_{N\to\infty}F^{(r)}_{N}(\gamma)= \phi(\beta) \mbox{ a.s.}\, .
\end{align}
If $\beta\ge0$, then we may take $\eta=1/(4\sqrt{k-1})$.
\end{lem}

\begin{lem}
\label{lemma:ComparisonBoundedUnnounded}
Let $G_N=(\AA_N,\bzero)$ be a sequence of weighted $k$-regular graphs, converging locally to the $k$-regular tree,
with edge weights $A_{ij}=\beta\in \R$ (either positive or negative)
for all $(i,j)\in E(G_N)$. If
$\lambda_{\max}(\AA_N^{\perp}),\lambda_{\min} (\AA_N^{\perp})\in I$ for some $N$-independent interval $I$ of length less than $1$,
then for any sequence with $\eps_N\to 0$ slow enough, we
have:
\begin{align}
\lim_{N\to\infty}\left| F_{b,N}(0;\eps_N) - F_{N}(k\beta) \right| =0
\, .
\end{align}
\end{lem}

We first note that, by the Alon-Boppana inequality for locally
tree-like graphs \cite{McKay1981Expected}, the spectral condition
implies $4\beta\sqrt{k-1}<1$, hence $\beta<1/(4\sqrt{k-1})$.
In particular, the hypothesis of Lemma
\ref{lemma:free-energy-random-ferro} holds. Hence the assumptions of
both Lemma \ref{lemma:free-energy-random-ferro} and Lemma
\ref{lemma:ComparisonBoundedUnnounded} hold.

Applying Theorem \ref{thm:main} and Lemma
\ref{lemma:free-energy-random-ferro}, we obtain:
\begin{align}
    \lim_{N\to\infty}F_{N}(k\beta)=
    \lim_{N\to\infty}F^{(r)}_{N}(k\beta)= \phi(\beta)\, .
 \end{align}
 Hence the claim of the corollary follows from Lemma
 \ref{lemma:ComparisonBoundedUnnounded}.

\subsubsection{Proof of Lemma
\ref{lemma:free-energy-random-ferro}}
\label{sec:proof-free-energy-random-ferro}

We will separately prove a lower bound and an upper bound. Throughout,
assume $|\beta|<1/(4\sqrt{k-1})$.

\paragraph{Lower bound.} For simplicity of
exposition we assume $N$ even, but the argument is the same for odd
$N$. For any $\gamma\in\R$ and $\eps_N =1/N$
\begin{align}
 F^{(r)}_{N}(\gamma)&\ge F^{(r)}_{b,N}(\gamma;\eps_N)
 =F^{(r)}_{b,N}(0;\eps_N)\, ,
\end{align}
and therefore it is sufficient to lower bound
$F^{(r)}_{b,N}(0;\eps_N)$. If $\beta\ge 0$, this is the free energy of
the ferromagnetic model with magnetization constrained to be equal to
zero, and \cite[Theorem 1.1]{GheissariPerkinsYap2025} implies
\begin{align}
    \lim_{N\to\infty}F^{(r)}_{b,N}(0;\eps_N) =\phi(\beta)\, .
\end{align}
If $\beta<0$, the same limit holds by a second-moment argument
analogous to that of Section~\ref{sec:proof-spin-glass}. Indeed, the
rate-function analysis in the upper bound below shows that
$q\mapsto\psi(q;\beta,0)$ is uniquely maximized at $q=0$, and
$|\beta|<1/(4\sqrt{k-1})$ yields
$(k-1)(\tanh\beta)^2\le(k-1)\beta^2<1/16<1$, so the second-moment
method applies upon replacing $\cosh(2\beta_{\sSG})$ by $e^{2\beta}$
in Eq.~\eqref{eq:spin-glass-constrained-expectation}. (The resulting
formulae for the first and second moments are valid for $\beta$ of
either sign.) We omit the routine details.

\paragraph{Upper bound.}
We obtain an upper bound by computing the expected partition function.
It is convenient to write $Z^{(r)}_{N}(q;\gamma)$ for the partition
function restricted to configurations with $\<\bsigma,\bone\>= Nq$. We
have, letting $\om=N-m$, $L=Nk/2$,
\begin{align}
    \E Z^{(r)}_{N}(q;0)&= \frac{1}{\Pair(Nk)} \binom{N}{m}
    \sum_{\ell=0}^{k m\wedge k\om}\binom{km}{\ell} \binom{k\om}{\ell}
    \ell! \Pair(km-\ell)\Pair(k\om-\ell) e^{\beta(L-2\ell)}\, ,\\ &m =
    \frac{N}{2}(1+q)\, ,
\end{align}
where the sum is restricted to the values of $\ell$ such that
$km-\ell$ is even. This expression is valid for $\beta\in\R$ of either
sign. When $\beta\ge 0$, it matches the expectation in Section
\ref{sec:proof-spin-glass} up to a change of variables: comparing with
Eq.~\eqref{eq:spin-glass-constrained-expectation}, if one replaces
$\cosh(2\beta_{\sSG})$ by $e^{2\beta}$ (we use the subscript $\sSG$ to
distinguish the two parameters), then the spin-glass expectation
exceeds the ferromagnetic one by the deterministic factor $2^Ne^{\beta
L}$. In all cases, the same saddle-point / local CLT analysis yields
\begin{align}
    \E Z^{(r)}_{N}(q;0)&= C\, N^{-1/2} \, e^{N\psi(q;\beta,0)}
    \big(1+o_N(1)\big)\, ,
\end{align}
where, for $\beta\ge 0$, one has
$\psi(q;\beta,0)=\phi_2(q;\beta_{\sSG})-\log 2-\beta k/2$, with
$\phi_2$ defined in Eq.~\eqref{eq:spin-glass-phi2-def} and
$\beta_{\sSG}$ determined by $\cosh(2\beta_{\sSG})=e^{2\beta}$. Since
\begin{align}
Z^{(r)}_{N}(q;\gamma)= e^{-N\gamma q^2/2}Z^{(r)}_{N}(q;0)\, ,
\end{align}
we get
\begin{align}
    \E Z^{(r)}_{N}(q;\gamma)&= C\, N^{-1/2} \,
    e^{N\psi(q;\beta,\gamma)} \big(1+o_N(1)\big)\, .
\end{align}
Using Eq.~\eqref{eq:spin-glass-phi2-def} (or the direct saddle-point
analysis when $\beta<0$), we obtain
\begin{align}
\psi(q;\beta,\gamma) &= -(k-1) \entro\!\left(\frac{1-q}{2}\right)
+\frac{k}{2} \log \tM(h_q,\beta) - \frac{k}{2}h_qq-\frac{1}{2}\gamma
q^2\, ,\\ \tM(h,\beta) &:= 2e^{\beta} \cosh h+2e^{-\beta}\, ,
\end{align}
where $h_q$ satisfies
\begin{equation}\label{eq:ferro-hq}
     q = \frac{e^{2\beta}\sinh h_q}{1+e^{2\beta}\cosh h_q}\, .
\end{equation}
We have $\psi(0;\beta,\gamma)= \phi(\beta)$ and
$\psi(-q;\beta,\gamma)=\psi(q;\beta,\gamma)$. We claim that there
exists $\eta=\eta(\beta,k)>0$ such that if $|\gamma-k\beta|<\eta$,
then $\partial_q\psi(q;\beta,\gamma)<0$ for all $q\in(0,1]$. This
implies
\begin{align}
    \E Z^{(r)}_{N}(\gamma)\le C N^{1/2} e^{N\phi(\beta)}\, .
\end{align}
Hence, for any $\eps>0$, Markov's inequality yields
$\P\bigl(F^{(r)}_{N}(\gamma)\ge \phi(\beta)+\eps\bigr) \le C
N^{1/2}e^{-N\eps}$, which is summable. Borel--Cantelli therefore gives
$\limsup_{N\to\infty} F^{(r)}_{N}(\gamma) \le\phi(\beta)$ almost
surely, completing the upper bound.

It remains to prove the claim on $\partial_q\psi$. Using
Eq.~\eqref{eq:ferro-hq}, we have
\begin{align}
    \frac{\de}{\de q}\psi(q;\beta,\gamma) &=
    \frac{k-1}{2}\log\frac{1+q}{1-q}-\frac{k}{2} h_q-\gamma q\, .
\end{align}
Set $c:=e^{2\beta}$ and define
\begin{equation}
    Q(h):=\frac{c\sinh h}{1+c\cosh h}\, .
\end{equation}
Then $q=Q(h)$ defines a strictly increasing bijection from $\R$ onto
$(-1,1)$, and
\begin{equation}\label{eq:ferro-qh}
    \frac{1+q}{1-q}=\frac{1+c\,e^{h}}{1+c\,e^{-h}} \, .
\end{equation}
By symmetry it is enough to treat $q\in(0,1]$, for which
$h=Q^{-1}(q)\ge 0$. Therefore $\partial_q\psi(q;\beta,\gamma)<0$ if
and only if, for $h=Q^{-1}(q)$,
\begin{equation}\label{eq:ferro-g}
    g(h):=kh-(k-1)\log\frac{1+ce^h}{1+ce^{-h}} +2\gamma Q(h)>0\, .
\end{equation}
We have $g(0)=0$, so it is enough to show $g'(h)>0$ for all $h\ge 0$.
Differentiating \eqref{eq:ferro-g} and writing $x:=\cosh h\ge 1$ gives
\begin{equation}\label{eq:ferro-gprime}
    g'(h)=k-2c(k-1)A(x)+2\gamma\, c\,B(x)\, ,
\end{equation}
where
\begin{equation}
    A(x)=\frac{x+c}{1+c^2+2c x}\, , \qquad B(x)=\frac{x+c}{(1+cx)^2}\,
    .
\end{equation}
After some simple algebra, we obtain
\begin{equation}\label{eq:ferro-gprime-alt}
g'(h) = 1-(k-1)\frac{c^2-1}{1+c^2+2cx} +2\gamma\, c\,
\frac{x+c}{(1+cx)^2}\, .
\end{equation}
Write $\gamma=k\beta+\delta$. We first treat the case $\delta=0$, and
then absorb a small perturbation $\delta$.

\paragraph*{The case $\gamma=k\beta$ and $\beta\ge
0$.} Here $c=e^{2\beta}\ge 1$. Since $B$ is decreasing on $[1,\infty)$
with $B(1)=1/(1+c)$, and $1+c^2+2cx\ge 2c(x+1)$ together with
$(1+cx)^2\le c^2(x+1)^2$,
\begin{align}
(k-1)\frac{c^2-1}{1+c^2+2cx} &\le (k-1)\frac{c^2-1}{2c(x+1)}\, , \\
2k\beta\, c\, \frac{x+c}{(1+cx)^2} &\ge \frac{2k\beta}{c(x+1)}\, .
\end{align}
Substituting into Eq.~\eqref{eq:ferro-gprime-alt} we obtain
\begin{align}
g'(h) &\ge 1+\frac{1}{c(x+1)}\Bigl(2k\beta-\frac{(k-1)(c^2-1)
}{2}\Bigr) = 1+\frac{\Theta}{c(x+1)}\, ,
\end{align}
where $\Theta:=2k\beta-(k-1)(c^2-1)/2$. If $\Theta\ge 0$ then
$g'(h)\ge 1>0$. We can therefore assume that $\Theta<0$. Since $x+1\ge
2$, we have $\Theta/(c(x+1))\ge \Theta/(2c)$, and therefore
\begin{align}
g'(h) &\ge 1+\frac{\Theta}{2c} = 1+k\beta\,
e^{-2\beta}-\frac{k-1}{2}\sinh(2\beta)\, ,
\end{align}
where we used $(c^2-1)/c=c-c^{-1}=2\sinh(2\beta)$. Using
$e^{-2\beta}\ge 1-2\beta$ and $\sinh y\le y+(y^3/6)\cosh y$ with
$y=2\beta$,
\begin{align}
1+k\beta e^{-2\beta}-\frac{k-1}{2}\sinh(2\beta) &\ge
1+k\beta(1-2\beta)-(k-1)\beta-\frac{2(k-1)}{3} \beta^3\cosh(2\beta)
\nonumber\\ &= 1+\beta-2k\beta^2-\frac{2(k-1)}{3}
\beta^3\cosh(2\beta)\, .
\end{align}
Since $\beta<1/(4\sqrt{k-1})$, we have $2k\beta^2<k/(8(k-1))\le 3/16$
(using $k\ge 3$) and $\beta<1/4$, hence $\cosh(2\beta)<\cosh(1/2)\le
2$ and
\begin{equation}
\frac{2(k-1)}{3}\beta^3\cosh(2\beta) < \frac{4(k-1)}{3}\,\beta^3 <
\frac{4(k-1)}{3}\cdot\frac{1}{64\,(k-1)^{3/2}}
=\frac{1}{48\sqrt{k-1}}\le\frac{1}{48\sqrt{2}}\, .
\end{equation}
Thus
\begin{equation}
1+k\beta e^{-2\beta}-\frac{k-1}{2}\sinh(2\beta) >
1-\frac{3}{16}-\frac{1}{48\sqrt{2}} > \frac34 >0\, .
\end{equation}
This proves $g'(h)>0$ for all $h\ge 0$ when $\gamma=k\beta$ and
$\beta\ge 0$.

\paragraph*{The case $\gamma=k\beta$ and
$\beta<0$.} Write $\beta=-\alpha$ with $0<\alpha<1/(4\sqrt{k-1})$, so
$c=e^{-2\alpha}\in(e^{-1/2},1)$. In particular $c>1/2$, and the same
computation as above shows that $B$ is decreasing on $[1,\infty)$.
Likewise $A$ is decreasing on $[1,\infty)$. From
\eqref{eq:ferro-gprime} with $\gamma=-k\alpha$,
\begin{equation}
g'(h)=k-2c(k-1)A(x)-2k\alpha\, c\,B(x)\, .
\end{equation}
Both $A$ and $B$ are decreasing, so $g'(h)$ is increasing in $h$.
Hence it is enough to check $g'(0)>0$. Using
$(c-1)/(c+1)=\tanh\beta=-\tanh\alpha$ and
$c/(1+c)=1/(1+e^{2\alpha})=(1-\tanh\alpha)/2$, we get
\begin{align}
g'(0) &= 1-(k-1)\frac{c-1}{c+1}+2k\beta\frac{c}{1+c} =
1+(k-1)\tanh\alpha-k\alpha(1-\tanh\alpha) \nonumber\\ &=
1-k\alpha+\bigl(k(1+\alpha)-1\bigr)\tanh\alpha\, .
\end{align}
Since $\tanh\alpha\ge\alpha-\alpha^3/3$,
\begin{align}
g'(0) &\ge 1-k\alpha+\bigl(k(1+\alpha)-1\bigr)
\bigl(\alpha-\alpha^3/3\bigr) =
1-\alpha+k\alpha^2-\frac{k(1+\alpha)-1}{3} \alpha^3\, .
\end{align}
Using $\alpha<1/(4\sqrt{k-1})\le 1/(4\sqrt{2})$ and $k\ge 3$, we have
$1-\alpha>1-1/(4\sqrt{2})>3/4$, while $k\alpha^2\le k/(16(k-1))\le
3/32$ and the cubic remainder is at most $1/(48\sqrt{2})$, whence
$g'(0)>1/2>0$. Thus $g'(h)\ge g'(0)>0$ for all $h\ge 0$.

\paragraph*{Perturbing $\gamma$.}
From \eqref{eq:ferro-gprime},
\begin{equation}
\Bigl|g'(h)\big|_{\gamma=k\beta+\delta}-g'(h)\big|
_{\gamma=k\beta}\Bigr| = \bigl|2\delta\, c\, B(x)\bigr| \le
2|\delta|\,\frac{c}{1+c}\le 2|\delta|\, .
\end{equation}
If $\beta\ge 0$, the estimates above give
$g'(h)\big|_{\gamma=k\beta}>3/4$ for all $h\ge 0$, so
$g'(h)>3/4-2|\delta|$ whenever $\gamma=k\beta+\delta$. Taking
$\eta:=1/(4\sqrt{k-1})$ yields $2\eta=1/(2\sqrt{k-1})\le
1/(2\sqrt{2})<1/2$, hence $g'(h)>1/4>0$. If $\beta<0$, we have
$g'(h)\big|_{\gamma=k\beta}\ge g'(0)>1/2$, so the same bound gives
$g'(h)>0$ whenever $|\delta|<\eta:=1/8$. In either case,
$\partial_q\psi(q;\beta,\gamma)<0$ for all $q\in(0,1]$ as soon as
$|\gamma-k\beta|<\eta$, as claimed.

\subsubsection{Proof of Lemma
\ref{lemma:ComparisonBoundedUnnounded}}
\label{sec:proof-comparison-bounded-unbounded}

Under the assumptions of the lemma, we have
$|\lambda_{\max}(\AA_N^{\perp})|,|\lambda_{\min} (\AA_N^{\perp})|\ge
2|\beta|\sqrt{k-1}-o_N(1)$ by the Alon-Boppana inequality for locally
tree-like graphs \cite{McKay1981Expected}. Therefore, we have
$|\beta|<1/(4\sqrt{k-1})$.

For any $\eta_0\in (0,1/(4\sqrt{k-1})$, set $\gamma=k\beta-\eta_0$ and
note that $\AA_N(\gamma) =\AA_N^{\perp}+\eta_0\bone\bone^{\sT}/N$, and
therefore $\AA_N$ also satisfies the spectral condition of Theorem~\ref{thm:main}. Applying it along with Lemma~\ref{lemma:left-convergence-locally-tree-like} and Lemma
\ref{lemma:free-energy-random-ferro}, we have
%
%
%
\begin{align}
\lim_{N\to\infty}F_{N}(\gamma)= \lim_{N\to\infty}F^{(r)}_{N}(\gamma)=
\phi(\beta)\, .
\end{align}
As a consequence, for any $\eps_N>0$, we have
\begin{align*}
Z_{u,N}(k\beta;\eps_N)&\le \sum_{\substack{\bsigma\in\{-1,1\}^N\\
|\<\bsigma,\bone\>|> N\eps_N}} e^{\<\bsigma, \AA_N(k\beta)\bsigma\>/2}
e^{\eta_0 \<\bsigma,\bone\>^2/2N-N\eta_0\eps_N^2/2}\\ &\le
Z_{u,N}(k\beta-\eta_0;\eps_N) e^{-N\eta_0\eps_N^2/2} \\ &\le
\exp\Big\{N\phi(\beta)-\frac{N}{2} \eta_0\eps_N^2+o(N)\Big\}\, .
\end{align*}
Hence, there exists a sequence $\eps_N\to 0$ such that
\begin{align}
    \lim_{N\to\infty}\frac{Z_{b,N}(k\beta;\eps_N)} {Z_N(k\beta)}= 1-
    \lim_{N\to\infty}\frac{Z_{u,N}(k\beta;\eps_N)} {Z_N(k\beta)} =1\,
    .
\end{align}
Therefore
\begin{align}
    \lim_{N\to\infty}\big| F_{b,N}(k\beta;\eps_N) - F_{N}(k\beta)
    \big| =0\, .\label{eq:comparison-bounded-unbounded-1}
\end{align}
Finally we notice that
\begin{align}
    \lim_{N\to\infty} \big| F_{b,N}(k\beta;\eps_N) - F_{b,N}(0;\eps_N)
    \big| \le \lim_{N\to\infty} \frac{1}{2}k|\beta|\eps_N^2= 0 \,
    .\label{eq:comparison-bounded-unbounded-2}
\end{align}
The proof is completed by combining
Eqs.~\eqref{eq:comparison-bounded-unbounded-1} and
\eqref{eq:comparison-bounded-unbounded-2}.

\subsection{Proof of Corollary
\ref{cor:Antiferromagnetic}}

By Lemma \ref{lemma:ComparisonBoundedUnnounded}, and Lemma
\ref{lemma:free-energy-random-ferro}, for all $\eps_N$ vanishing slow
enough, we have 
\begin{align}
    \lim_{N\to\infty}F_{b,N}(0;\eps_N)=
    \lim_{N\to\infty}F_{N}(k\beta)= \phi(\beta)\, .\label{eq:Balanced-Not-Balanced}
\end{align}
On the other hand,  recalling the definition of $Z_{u,N}$ from Eq.~\eqref{eq:ZuDef},
\begin{align*}
    Z_{u,N}(0;\eps_N&)\le e^{Nk\beta\eps_N^2/2} Z_{u,N}(k\beta;\eps_N)
    \\ & \le e^{Nk\beta\eps_N^2/2} Z_N(k\beta) \\ & \le
    e^{-Nk|\beta|\eps_N^2/2+o(N)} Z_{b,N}(0;\eps_N) \, ,
\end{align*}
where in the last step we used Eq.~\eqref{eq:Balanced-Not-Balanced}.
Therefore
\begin{align*}
\lim_{N\to\infty}F_N(0) &= \lim_{N\to\infty}F_{b,N}(0;\eps_N)+
\lim_{N\to\infty}\frac{1}{N} \log
\Big(1+\frac{Z_{u,N}(0;\eps_N)}{Z_{b,N}(0;\eps_N)} \Big)\\ &=
\phi(\beta)+ \lim_{N\to\infty}\Big(
-\frac{1}{2}k|\beta|\eps_N^2+o(1)\Big)\\ &=\phi(\beta)\, .
\end{align*}
This completes the proof.

\subsection{Proof of Proposition
\ref{prop:NecessitySpectral}}

Let $G^{(1)}_N=(\AA^{(1)}_N,\bzero)$ be a sequence of random
$k$-regular graphs with edge weights $A^{(1)}_{ij}=\beta<0$ for all
$(i,j)\in E(G^{(1)}_N)$, and $G^{(2)}_N=(\AA^{(2)}_N,\bzero)$ be a
sequence of bipartite random-$k$ regular graphs, with the same edge
weights. We draw the random graphs according to the respective
configuration models and assume for simplicity $N$ even (otherwise,
add an extra degree-$0$ node to the graph).

We take
\begin{align}
\atanh\Big(\frac{1}{k-1}\Big)<|\beta|< \frac{1}{4\sqrt{k-1}}\, .
\end{align}
Under the second condition, Corollary~\ref{cor:Antiferromagnetic} 
yields
\begin{align}
\lim_{N\to\infty}F(G^{(1)}_N)= \phi(\beta)= \log 2+\frac{k}{2}\log
\cosh(\beta)\, .
\end{align}

On the other hand, since $G^{(2)}_N$ is bipartite, the corresponding
free energy density is equal to the one of the same model with
positive weights $|\beta|$. This is computed in \cite[Theorem
2.4]{dembo2010ising}, which specializes to
\begin{align}
    \lim_{N\to\infty}F(G^{(2)}_N)&= \sup_{h\ge 0}\Psi(h;\beta)\, ,\\
    \Psi(h;\beta) &= \frac{k}{2}\log\cosh \beta
    -\frac{k}{2}\log\bigl(1+\tanh(|\beta|) \tanh^2(h)\bigr)\\ &\quad
    +\log\Bigl( \bigl(1+\tanh(|\beta|)\tanh(h)\bigr)^k
    +\bigl(1-\tanh(|\beta|)\tanh(h)\bigr)^k \Bigr)\, .
\end{align}
A simple calculation yields (defining $\theta=\tanh|\beta|$)
\begin{align}
    \Psi(h;\beta) =\phi(\beta)+\frac{k\theta}{2}\big[(k-1)
    \theta-1\big]\, h^2+O(h^4)\, .
\end{align}
Therefore, under the current choice for $\beta$, we have
\begin{align}
    \lim_{N\to\infty}F(G^{(2)}_N)&>\phi(\beta)\, .
\end{align}

On the other hand, both $G^{(1)}_N$ and $G^{(2)}_N$ converge locally weakly to the $k$-regular tree with edge weights $\beta$. The spectrum of $\AA^{(\ell)}_N$ is contained in $[-k|\beta|,k|\beta|]$ and $|\beta|$ can be chosen arbitrarily close to $\atanh\left(\frac1{k-1}\right)$ yielding a spectral width less than $2+\delta$ for any $\delta>0$ with a suitably large choice of $k$.
%

\newcommand{\etalchar}[1]{$^{#1}$}
\providecommand{\bysame}{\leavevmode\hbox to3em{\hrulefill}\thinspace}
\providecommand{\MR}{\relax\ifhmode\unskip\space\fi MR }
\providecommand{\MRhref}[2]{%
  \href{http://www.ams.org/mathscinet-getitem?mr=#1}{#2}
}
\providecommand{\href}[2]{#2}


\end{document}